\documentclass[12pt,leqno,fleqn]{amsart}
\usepackage{amsmath,amstext,amsthm,amssymb,amsxtra}
\usepackage{txfonts,pxfonts} 
\usepackage[colorlinks,citecolor=red,pagebackref,hypertexnames=false]{hyperref}
\usepackage[backrefs,author-year]{amsrefs}
\usepackage{pgf,tikz}
\usetikzlibrary{arrows}
\allowdisplaybreaks

\theoremstyle{plain} 
\newtheorem{lemma}[equation]{Lemma}
\newtheorem{proposition}[equation]{Proposition}
\newtheorem{theorem}[equation]{Theorem}

 \newtheorem{conjecture}[equation]{Conjecture}
\newtheorem{smallBallConjecture}[equation]{Small Ball Conjecture}

\newtheorem{talagrand}[equation]{Talagrand's Theorem}
\newtheorem{bg}[equation]{The  Simplest Instance of the Beck Gain}
\newtheorem{lp}[equation]{Littlewood-Paley Inequalities}
\newtheorem{discrep}[equation]{The $ L ^{\infty }$ Norm of Discrepancy Function Conjecture}
\newtheorem{roth}[equation]{K.~Roth's Theorem}
\newtheorem{smooth}[equation]{Smooth Small Ball Conjecture}
\newtheorem{sheet}[equation]{The Small Ball Conjecture for the Brownian Sheet}

\theoremstyle{definition}
\newtheorem{definition}[equation]{Definition}

\theoremstyle{remark}
\newtheorem{remark}[equation]{Remark}

\numberwithin{equation}{section}

\def\norm#1.#2.{\lVert#1\rVert_{#2}}
\def\Norm#1.#2.{\bigl\lVert#1\bigr\rVert_{#2}}
\def\NOrm#1.#2.{\Bigl\lVert#1\Bigr\rVert_{#2}}
\def\NORm#1.#2.{\biggl\lVert#1\biggr\rVert_{#2}}
\def\NORM#1.#2.{\Biggl\lVert#1\Biggr\rVert_{#2}}

\def\ip#1,#2,{\langle #1,#2\rangle}
\def\Ip#1,#2,{\bigl\langle#1,#2\bigr\rangle}
\def\IP#1,#2,{\Bigl\langle#1,#2\Bigr\rangle}

\def\mid{\,:\,}

\def\abs#1{\lvert#1\rvert}
\def\Abs#1{\bigl\lvert#1\bigr\rvert}
\def\ABs#1{\biggl\lvert#1\biggr\rvert}

\def\XXint#1#2#3{{\setbox0=\hbox{$#1{#2#3}{\int}$}
     \vcenter{\hbox{$#2#3$}}\kern-.5\wd0}}

\begin{document}
\title[Small Ball Inequality in All Dimensions]
{On the Small Ball Inequality in All Dimensions}
\author[D.~Bilyk]{Dmitriy Bilyk}
\address{School of Mathematics \\ Georgia Institute of Technology \\ Atlanta GA 30030}

\email{bilyk@math.gatech.edu}

\author[M.\thinspace T.~Lacey]{Michael T. Lacey}
\address{School of Mathematics \\ Georgia Institute of Technology \\ Atlanta GA 30030}

\email{lacey@math.gatech.edu}

\author[A. Vagharshakyan]{Armen Vagharshakyan}
\address{School of Mathematics \\ Georgia Institute of Technology \\ Atlanta GA 30030}

\email{armenv@math.gatech.edu}

\begin{abstract}
Let $ h_R$ denote an $ L ^{\infty }$ normalized Haar function
adapted to a dyadic rectangle $ R\subset [0,1] ^{d}$. We show that
for choices of coefficients $ \alpha (R)$, we have the following
lower bound on the $ L ^{\infty }$ norms of the sums of such
functions, where the sum is over rectangles of a fixed volume:
\begin{equation*}
 n ^{ \frac{d-1}{2}-\eta}  \NOrm \sum _{\abs{ R}= 2 ^{-n}} \alpha (R) h_R (x).
 L ^{\infty} ([0,1] ^{d}). \gtrsim 2^{-n} \sum _{\abs{ R}= 2 ^{-n}} \abs{\alpha
 (R)}
\,, \quad \textup{for some } \quad 0<\eta < \frac{1}{2}\,.
\end{equation*}
  The point
of interest is the dependence upon the logarithm of the volume of
the rectangles. With $n^{(d-1)/2}$ on the left above, the inequality   is trivial,
while it is conjectured that the inequality holds  with
$n^{(d-2)/2}$. This is known    in the
case of $ d=2$ \cite{MR95k:60049}, and a recent paper of two of the authors \cite{bl}
proves a  partial result towards the conjecture in three dimensions.
In this paper, we show that the argument of \cite{bl} can be extended to
arbitrary dimension.  We also prove related results in the subjects of  the Irregularity
of Distribution, and Approximation Theory.
The authors are unaware of any prior results on these questions in any
dimension  $ d\ge4$.
\end{abstract}

\maketitle

\section{The Small Ball Conjectures} 

In this paper we will prove results in dimension four and higher in
three separate areas, Number Theory, Approximation Theory, and
Probability Theory: (a) the theory of Irregularities of
Distribution, (b) the Kolmogorov Entropy of spaces of functions with
bounded mixed derivative, and (c) Small Deviation Inequalities  for the Brownian
Sheet. As far as the authors are aware, these are the first  results
on these questions which provide more information than that given by
an average case analysis. Underlying these three results is a
central inequality, the \emph{Small Ball Inequality} for the Haar
functions, which we state here.  The related areas are addressed in
the next section.

 In one dimension, the class of dyadic intervals is $\mathcal D {} \coloneqq {}\{ [j2^k,(j+1)2^k)\mid j,k\in \mathbb Z\} $.
 Each dyadic interval has a left and right half, indicated below, which are also dyadic.  Define the
 Haar functions
 \begin{equation*}
h_I \coloneqq -\mathbf 1 _{I_{\textup{left}}}+ \mathbf 1 _{I_{\textup{right}}}
\end{equation*}
 Note that this is an $ L ^{\infty }$ normalization of these functions, which we will
 keep throughout this paper.

 In dimension $d $, a \emph{dyadic rectangle} is a product of dyadic intervals, thus an
element of
 $\mathcal D^d $.   We define a Haar function associated to $R $  to  be the product of the Haar functions associated
 with each side of $R $, namely
 \begin{equation*}
 h_{R_1\times\cdots\times R_d }(x_1,\ldots,x_d) {} \coloneqq {}\prod _{j=1}^d h _{R_j}(x_j).
 \end{equation*}
 This is the usual `tensor' definition.

 We will concentrate on rectangles with  fixed volume and
 consider a local problem.
 This is the `hyperbolic'
 assumption, that pervades the subject.
 Our concern is the following Theorem and Conjecture concerning a
 \emph{lower bound} on the $ L ^{\infty }$ norm of sums of hyperbolic Haar functions:

 \begin{smallBallConjecture} \label{small} For dimension $ d\ge 3$ we have the inequality
 \begin{equation}\label{e.Talagrand}
2 ^{-n} \sum _{\abs{ R}= 2 ^{-n} } \abs{ \alpha(R) } {}\lesssim{} n
^{\frac12(d-2) } \NOrm \sum _{\abs{ R}\ge 2 ^{-n} } \alpha (R) h_R
.\infty ..
\end{equation}
 \end{smallBallConjecture}

 Average case analysis --- that is passing through $ L ^2 $  --- shows that we always have
\begin{equation*}
2 ^{-n} \sum _{\abs{ R}= 2 ^{-n} } \abs{ \alpha(R) } {}\lesssim{} n
^{\frac12(d-1) } \NOrm \sum _{\abs{ R}\ge 2 ^{-n} } \alpha (R) h_R
.\infty ..
\end{equation*}
Namely, the constant on the right is bigger than in the conjecture by a factor
of $ \sqrt n$.
 We refer to this as the `average case estimate,' and refer to improvements over this
 as a `gain over the average case estimate.'
Random choices of coefficients $ \alpha (R)$ show that the Small Ball Conjecture is sharp.

In dimension $ d=2$, the Conjecture was resolved by \cite{MR95k:60049}.\footnote{This
result should be compared to \cite{MR0319933}, as well as \cites{MR637361,MR96c:41052}.}

 \begin{talagrand}\label{j.talagrand}
    For dimensions $d\ge2 $, we have
 \begin{equation}  \label{e.talagrand}
2 ^{-n} \sum _{\abs{ R}= 2 ^{-n} } \abs{ \alpha(R) } {}\lesssim{}
\NOrm \sum _{\abs R = 2 ^{-n}} \alpha (R) h_R .\infty ..
 \end{equation}
 Here, the sum on the right is taken over all rectangles with area \emph{at least } $ 2
 ^{-n}$.
 \end{talagrand}

The main result of this note is the next Theorem, which shows that
there is a gain over the trivial bound in the  Small Ball Conjecture
in dimensions $ d\ge 3$.   In dimension $ d=3$, this result was
proved in \cite{bl}. The three-dimensional result and its present
extension build
 upon
 the method devised by \cite{MR1032337}. As far as the authors are aware, this is
the first `gain over the average case bound'
 known in dimensions four and higher.

\begin{theorem}\label{t.d=3} In dimension $ d\ge 3$, there exists a number $\eta(d)>0$ such
that for all choices of coefficients $ \alpha (R)$, we have the
inequality
\begin{equation}\label{e.d=3}
 n ^{\frac{d-1}{2}-\eta (d) }  \NOrm \sum _{\abs{ R} \ge 2 ^{-n}} \alpha (R) h_R. \infty .
\gtrsim 2 ^{-n} \sum _{\abs{ R}= 2 ^{-n} } \abs{ \alpha(R) }\,.
\end{equation}
\end{theorem}

We take this Theorem as basic to our study, and use its proof  to
derive results on the  three other questions mentioned at the
beginning of the introduction.

The principal difficulty in three and higher dimensions is that two
dyadic rectangles of the same volume can share a common side length.  Beck
\cite{MR1032337} found a specific estimate in this case, an estimate that
is extended in \cite{bl}.  In this note, the main technical device is the
extension of this estimate, in the simplest instance, to arbitrary dimensions,
 see Lemma~\ref{l.SimpleCoincie}. This Lemma, and its extension to longer products
 Theorem~\ref{l.admissible<},
 is the main technical innovation of this paper.   The value of $ \eta $ that we can get out of this
line of reasoning appears to be of the order $d ^{-2} $, imputing
additional interest to the methods of proof used to improve this
estimate. Indeed, many aspects of our analysis are suboptimal, and the
most essential techniques necessary to optimize the arguments of this paper
are yet to be discovered.

\section{Related Results}

\subsection*{The $ L ^{\infty }$ Norm of the Discrepancy Function} 

In $d$ dimensions, take $\mathcal A_N$ to be $N$ points in the unit cube, and consider
the Discrepancy Function
\begin{equation}  \label{e.discrep}
D_N(x):=\sharp \mathcal A_N \cap [\vec 0,\vec x)-N \abs{ [\vec 0,\vec x)}\,.
\end{equation}
Here, $[\vec 0,\vec x)=\prod _{j=1}^d [0,x_j)$, that is a rectangle with antipodal corners being $\vec 0$ and $\vec x$.
Relevant norms of this function must tend to infinity,
in dimensions $2$ and higher.
 The  canonical result of this type
is the following estimate proved in \cite{MR0066435}.

\begin{roth}  \label{t.roth}We have the universal estimate
\begin{equation*}
\norm D_N. 2. \gtrsim (\log N) ^{(d-1)/2}\,,
\end{equation*}
with the implied constant only depending upon dimension.
\end{roth}

For all $ 1<p<\infty $, $ \norm D_N .p.$ admits the same lower
bound, a result in  \cite{MR0491574}.
The endpoint estimates of $ p=1,\infty $ are however much
harder, with definitive information known only in two dimensions.
The method of proof of this Theorem, and the $ L ^{p}$ variants can
be summarized as follows:  Fix $ 2N\le 2 ^{n}<N$, and just project
the Discrepancy Function onto the (hyperbolic) Haar functions $ \{
h_R\mid \lvert  R\rvert=2 ^{-n} \}$. By the Bessel inequality, this
provides a lower bound on the $ L ^{2}$ norm of $ D_N$. This same
method of proof, with the Littlewood-Paley inequalities replacing
the Bessel inequality, can be used to prove the $ L  ^{p}$ lower
bound, for $ 1<p<\infty $.  See \cite{MR903025}.

At $ L ^{\infty }$, guided by the sharpness of the Small Ball Conjecture, we pose
the Conjecture below, which represents a $ \sqrt {\log N}$ gain over the lower bound
proved by Roth.

\begin{discrep}
In dimension $ d\ge3$, we have the lower estimate valid for all point sets $ \mathcal A_N$.
\begin{equation*}
\norm D_N. \infty . \gtrsim (\log N) ^{d/2}\,.
\end{equation*}
\end{discrep}

In dimension $ d=2$, this is the Theorem of \cite{MR0319933}.  In dimension $ d=3$,
\cites{MR1032337,bl} give partial information about this conjecture.  In this paper,
we can prove the following result, which appears to be new in dimensions $ d\ge 4$.

\begin{theorem}\label{t.discrep}  In dimension $ d\ge 3$ there is a positive
$ \eta =\eta (d)>0$ for which we have the uniform estimate
\begin{equation*}
\norm D_N. \infty . \gtrsim (\log N) ^{(d-1)/2+\eta }\,.
\end{equation*}
\end{theorem}

The proof of this result follows easily from the method of proof of Theorem~\ref{t.d=3},
and will be presented below.

\subsection*{Metric Entropy of Mixed Derivative Sobolev Spaces}

While the special structure of the Haar functions can be exploited
to prove the Small Ball Conjecture, one would \emph{not} anticipate
that this special structure  is in fact essential to the Conjecture.
Thus, we formulate a smooth variant of the Small Ball Conjecture.

Fix a continuous non-constant function $ \varphi $, supported on $
[-1/2,1/2]$, and of mean zero.  For a dyadic interval $ I$, let
\begin{equation*}
\varphi _{I} (x)= \varphi \bigl( \tfrac {x- c (I)} {\lvert  I\rvert } \bigr)\,,
\end{equation*}
be a translation and rescaling of $ \varphi $ so that it is supported on $ I$.
Then, for a dyadic rectangle $ R=R_1 \times \cdots \times R_d$, set
\begin{equation*}
\varphi _{R} (x_1 ,\dotsc, x_d)=\prod _{j=1} ^{d} \varphi _{R_j} (x_j)\,.
\end{equation*}

\begin{smooth}\label{c.smooth}
 For dimension $ d\ge 3$ we have the inequality
 \begin{equation}\label{e.smoothTalagrand}
2 ^{-n} \sum _{\abs{ R}= 2 ^{-n} } \abs{ \alpha(R) } {}\lesssim{} n
^{\frac12(d-2) } \NOrm \sum _{\abs{ R}\ge 2 ^{-n} } \alpha (R)
\varphi _R .\infty ..
\end{equation}
The implied constant depends upon dimension $ d$ and $ \varphi $ only.
\end{smooth}

In this direction, we will prove a result in the same spirit as our Main Theorem.

\begin{theorem}\label{t.smooth}
Suppose  $ \varphi $ is continuous, supported on $ [-1/2,1/2]$,  of
mean zero, and such that $ \ip \varphi , h _{[-1/2,1/2]}, \neq 0$.
For dimension $ d\ge 3$, there is a positive
$ \eta = \eta (d)$ so that we have the inequality below
 \begin{equation}\label{e.smoothEtaTalagrand}
2 ^{-n} \sum _{\abs{ R}= 2 ^{-n} } \abs{ \alpha(R) } {}\lesssim{} n
^{\frac12(d-1) - \eta  } \NOrm \sum _{\abs{ R}\ge 2 ^{-n} } \alpha
(R) \varphi _R .\infty ..
\end{equation}
The implied constant depends upon $ \varphi $.
\end{theorem}

With this Theorem, we can establish new results on the metric entropy of
certain Sobolev spaces of functions with mixed derivative in certain $ L ^{p}$ spaces.
 In $d $ dimensions,
 consider the map
 \begin{equation*}
 \operatorname {Int}_d f(x_1,\cdots,x_d) {}\coloneqq{}\int_0 ^{x_1 }\!\!\cdots \!\!\int_0 ^{x_d } f(y_1,\cdots,y_d)\; dy_1\cdots
 dy_d.
 \end{equation*}
 We consider this as a map from $L ^{p }([0,1]^d) $ into $C([0,1]^d) $.
 Clearly, the image of $\operatorname {Int}_d $ consists of
 functions with $ L ^{p}$ integrable mixed partial derivatives.
 Let us set
 \begin{equation*}
 \operatorname {Ball}(MW ^{p} ([0,1] ^{d}))
 \coloneqq   \operatorname {Int}_d ( \{f \in L ^{p} ([0,1] ^{d})
 \mid \norm f.p. \lesssim 1\})\,.
\end{equation*}
That is, this is the image of the unit ball of $ L ^{p}$. This is the unit ball of the  space of
functions with mixed derivative in $ L ^{p}$.

These sets are  compact in in $C([0,1]^d) $, and it is of relevance
to quantify the compactness, through the device of \emph{covering
numbers}. For $0<\epsilon<1 $, set  $N(\epsilon, p, d) $ to be the
least number $N $ of  points $x_1,\cdots,x_N\in C([0,1]^d) $ so that
\begin{equation*}
 \operatorname {Ball}(MW ^{p} ([0,1] ^{d}))
\subset\bigcup _{n=1}^N \left( x_n+\epsilon B_\infty\right) .
\end{equation*}
Here, $ B _{\infty }$ is the unit ball of $ C ([0,1] ^{d})$.
The task at hand is to uncover the correct order of growth of these
numbers as $ \epsilon \downarrow 0$.
 The case of $ d=2$ below follows from Talagrand \cite{MR95k:60049},
 and the upper bound is known in full generality \cites{2000b:60195,MR96c:41052}.

\begin{conjecture}\label{c.n} For $ d\ge 2 $ one has the estimate
\begin{equation*}
\log N(\epsilon, p , d) \simeq  \epsilon ^{-1}
 (\log 1/\epsilon)^{d-1/2} \,,  \qquad \epsilon \downarrow 0\,.
\end{equation*}
\end{conjecture}

It is well known \cite{MR1005898} that results such as Theorem~\ref{t.smooth} can
be used to give new lower bounds on these covering numbers.

\begin{theorem}\label{t.N}  For $ 1\le p<\infty $, and $ d\ge 3$, there is
a $ \eta >0$ for which we have
\begin{equation*}
\log N (\epsilon , p ,d) \gtrsim \epsilon ^{-1} (\log 1/\epsilon ) ^{d-1+\eta }\,.
\end{equation*}
\end{theorem}

  We have concentrated on the case
of one mixed derivative, but various results on fractional
derivatives are also interesting.  See for instance
\cite{MR2003m:60131}, and \cite{MR1777539}.

\subsection*{The Small Ball inequality for the Brownian Sheet}

Perhaps, it is worthwhile to explain the nomenclature `Small Ball'
at this point. The name comes from the probability theory. Assume
that $X_t: T \rightarrow \mathbb R$ is a canonical Gaussian process
indexed by a set $T$. The \emph{Small Ball Problem} is concerned
with estimates of $\mathbb P ( \sup_{t\in T} |X_t|< \varepsilon )$
as $\varepsilon$ goes to zero, i.e the probability that the random
process takes values in an $L^\infty$ ball of small radius. The
reader is advised to consult a paper by   Li and Shao
\cite{MR1861734} for a survey of this type of questions.
A particular question of interest to us deals with the Brownian Sheet,
that is, a centered Gaussian process indexed by the points in the
unit cube $[0,1]^d$ and characterized by the covariance relation
$\mathbb E X_s \cdot X_t = \prod_{j=1}^d \min (s_j,t_j)$.

Kuelbs and Li \cite{MR94j:60078} have discovered a tight connection
between the Small Ball probabilities and the properties of the
reproducing kernel Hilbert space corresponding to the process, which
in the case of the Brownian Sheet is $WM^2([0,1]^d)$, the space
described in the previous subsection. Their result, applied to the setting
of the Brownian sheet in
\cite{2000b:60195}, states that

\begin{theorem}
In dimension $d\ge 2$, as $\varepsilon \downarrow 0$ we have
$$-\log
\mathbb P (\norm B.C([0,1]^d). < \varepsilon ) \simeq
{\varepsilon}^{-2} ( \log 1/\varepsilon )^{\beta} \quad \textup{
iff} \quad \log N(\varepsilon) \simeq {\varepsilon}^{-1} ( \log
1/\varepsilon )^{\beta/2}.$$
\end{theorem}

Thus, in agreement with Conjecture \ref{c.n}, the conjectured form
of the aforementioned probability in this case is the following:

\begin{sheet}\label{c.sheet}
In dimensions $d\ge 2$, for the Brownian Sheet $B$ we have $$-\log
\mathbb P (\norm B.C([0,1]^d). < \varepsilon ) \simeq
{\varepsilon}^{-2} ( \log 1/\varepsilon )^{2d-1}, \quad \varepsilon
\downarrow 0 .$$
\end{sheet}
In dimension $d=2$, this conjecture has been resolved by Talagrand
in the already cited paper \cite{MR95k:60049}, in which he actually
proved Conjecture \ref{c.smooth} for a specific function $\varphi$
and used it to deduce the lower bound in the inequality above.\footnote{The work
of Talagrand bears strong similarities to the prior work of  \cite{MR0319933}
and  \cite{MR637361}.  The argument of Talagrand
was subsequently clarified by   \cite{MR96c:41052}, and  \cite{MR1777539}.
}
 In higher dimensions, the upper bounds are established, see \cite{2000b:60195},
 and the
previously known lower bounds miss the conjecture by a single power
of the logarithm.

Theorem \ref{t.N} can be translated into the following result on the
Small Ball Probability for the Brownian Sheet:

\begin{theorem}\label{t.sheet}
In dimensions $d\ge 3$, there exists $\eta>0$ such that for the
Brownian Sheet $B$ we have
$$-\log \mathbb P (\norm B.C([0,1]^d). < \varepsilon ) \gtrsim
{\varepsilon}^{-2} ( \log 1/\varepsilon )^{2d-2+\eta}, \quad
\varepsilon \downarrow 0 .$$
\end{theorem}

\section{Notations and Littlewood-Paley Inequality} 
Let $\vec r\in \mathbb N^d$ be a partition of $n$, thus $\vec r=(r_1
,\dotsc,  r_ d)$, where the  $r_j$ are nonnegative integers and
$\abs{ \vec r} \coloneqq \sum _t r_t=n$, which we refer to as the
\emph{length of the vector $ \vec r$}. Denote all such vectors as $
\mathbb H _n$. (`$ \mathbb H $' for `hyperbolic.') For vector $ \vec
r $ let $ \mathcal R _{\vec r} $ be all dyadic rectangles $ R$ such
that for each coordinate $ k$, $ \lvert  R _k\rvert= 2 ^{-r_k} $.

\begin{definition}\label{d.rfunction}
We call a function $f$ an \emph{$\mathsf r$ function  with parameter $ \vec r$ } if
\begin{equation}
\label{e.rfunction} f=\sum_{R\in \mathcal R _{\vec r}}
\varepsilon_R\, h_R\,,\qquad \varepsilon_R\in \{\pm1\}\,.
\end{equation}
A fact used without further comment is that $ f _{\vec r} ^2 \equiv 1$.
\end{definition}

As it has been already pointed out, the principal difficulty in
three and higher dimensions is that the product of Haar functions is
not necessarily a Haar function. On this point, we have the
following 

\begin{proposition}\label{p.productsofhaars}
Suppose that  $R_1,\ldots,R_k$ are rectangles such that there is no
choice of $1\le j<j'\le k$ and no choice of coordinate $1\le{} t\le
d$ for which we have $R _{j,t}=R _{j',t}$.  Then, for a choice of
sign $\varepsilon\in \{\pm1\}$ we have
\begin{equation}  \label{e.productsofhaars}
\prod _{j=1}^k h_R=\varepsilon h_S, \qquad S=\bigcap _{j=1}^k R_k \,.
\end{equation}
\end{proposition}

\begin{proof}
Expand the product as
\begin{equation*}
\prod _{m=1}^\ell  h_{R_m} (x_1,\dotsc, x_d) = \prod _{m=1}^\ell
 \prod _{t=1} ^{d} h_{R_{m,t}} (x_t)\,.
\end{equation*}
Here $ \varepsilon _m\in \{\pm 1\}$.  Our assumption is that for
each $ t$, there is exactly one choice of $ 1\le m_0\le \ell $ such
that $ R _{m_0,t}=S_t$.  And moreover, since the minimum value of $
\abs{ R _{m,t}}$ is obtained exactly once, for $ m\neq m_0$, we have
that $ h _{R _{m,t}}$ is constant on $ S_t$.  Thus, in the $ t$
coordinate, the product is
\begin{equation*}
  h _{S_t} (x_t) \prod _{1\le m\neq m_0\le \ell }
h _{R _{m,t}} (S_t)\,.
\end{equation*}
This proves our Lemma.
\end{proof}

\begin{remark}
It is also a useful observation, that the product of Haar functions will
have mean zero if the minimum value of $ \abs{ R _{m,t}}$ is unique
for at least one coordinate $t$.
\end{remark}

\begin{definition}\label{d.distinct}  For vectors $ \vec r_j \in \mathbb N ^{d}$,
say that $ \vec r_1,\dotsc,\vec r_J$ are \emph{strongly distinct }
iff for coordinates $ 1\le t\le d$ the integers $ \{ r _{j,t}\mid
1\le j \le J\}$ are distinct.  The product of strongly distinct $
\mathsf r$ functions is also an $ \mathsf r$ function, which follows
from `the product rule' \eqref{p.productsofhaars}.
\end{definition}

The $\mathsf r$ functions we are interested in are:
\begin{equation}\label{e.fr}
f _{\vec r} \coloneqq \sum _{R\in \mathcal R _{\vec r}}
\operatorname {sgn} (\alpha  (R)) \, h_R \,.
\end{equation}

\bigskip

We recall some Littlewood-Paley inequalities, which are
 standard, and so we omit proofs.
\begin{lp}  In one dimension, we have the inequalities
\begin{equation}\label{e.lp}
\NOrm \sum _{I\subset \mathbb R } a_I h_I (\cdot) .p.
\lesssim
\sqrt p \NORm \biggl[ \sum _{I\subset \mathbb R }  { a_I ^2 }
\mathbf 1 _{I} (\cdot )\biggr] ^{1/2} .p.\,,
\qquad 2<p<\infty \,.
\end{equation}
Moreover, these inequalities continue to hold in the case where
the coefficients $ a_I $ take values in a Hilbert space $ \mathcal H$.
\end{lp}

The growth of the constant is essential for us, in particular the
factor $ \sqrt p$ is, up to a constant, the best possible in this
inequality.  See \cites{MR1439553,MR1018577}.  That these
inequalities hold for Hilbert space valued sums is imperative for
applications to higher dimensional sums of Haar functions. The
relevant inequality is as follows.

\begin{theorem}\label{t.LP} We have the inequalities below
for hyperbolic sums of $ \mathsf r$ functions in   dimension $ d\ge 3$.
\begin{equation}\label{e.LP}
\NOrm \sum _{\lvert  \vec r\rvert =n} f _{\vec r} .p.
\lesssim (p n) ^{ (d-1)/2}\,, \qquad 2<p<\infty \,.
\end{equation}
\end{theorem}

%
%
%
%
%
%

We recall a vector valued Harmonic Analysis  inequality.

\begin{proposition}\label{p.condExpect}
Let $ \mathcal F_j$ be a sigma field generated
by dyadic rectangles in dimension $ 2$.  We then have
\begin{equation}\label{e.FS}
\NORm \Biggl[ \sum _{j}   \mathbb E (\varphi _j \,\vert\, \mathcal F_j) ^2
\Biggr] ^{1/2} .p.
\lesssim  p
\NORm \Biggl[ \sum _{j} \varphi _j ^2
\Biggr] ^{1/2} .p. \,, \qquad 2<p<\infty \,.
\end{equation}

\end{proposition}
\begin{proof}
This is  one of many examples of a  vector valued inequality in the
Harmonic Analysis literature.  This particular inequality admits a simple
proof by duality, recalled here for convenience.

Since $ p>2$, we can appeal to a duality argument.
We can choose $ g\in L ^{(p/2)'}$ of norm one so that
\begin{align*}
\NORm
\sum _{j}   \mathbb E (\varphi _j \,\vert\, \mathcal F_j)^2
.p/2.
& =
\sum _{j}   \ip \mathbb E (\varphi _j \,\vert\, \mathcal F_j)^2  , g,
\\
& \le
\sum _{j}   \ip \mathbb E (\varphi _j ^2  \,\vert\, \mathcal F_j)  , g,
\\
& =
\sum _{j}   \ip\varphi _j ^2    , \mathbb E (g  \,\vert\, \mathcal F_j) ,
\\
& \le \sum _{j} \ip \varphi _j ^2 , \operatorname M g ,
\\
& \le \NORm \sum _{j}  \varphi _j ^2. p/2.  \norm  \operatorname M g . (p/2)' .
\\
& \lesssim  ( (p/2)'-1) ^{-2} \NORm \sum _{j}  \varphi _j ^2. p/2. \,.
\end{align*}
Here we have used Jensen's inequality and the self-duality of the
conditional expectation operators.  The operator $ \operatorname M
g$ is the (strong) maximal function on the plane, namely
\begin{equation*}
\operatorname M g (x) =\sup _{R} \frac {\mathbf 1_{R}} {\lvert
R\rvert } \int_R \lvert  g (y)\rvert\; dy,
\end{equation*}
where the supremum is over all dyadic rectangles $ R$.  This maps $ L ^{q}$ into $ L ^{q}$
for all $ 1<q<\infty $, an inequality appealed to in the last line of the display above.
Moreover, it is well known that  the norm of the operator behaves as
\begin{equation*}
\norm \operatorname M . q \to q. \lesssim (q-1) ^{-2} \,, \qquad 1<q<2\,.
\end{equation*}

\end{proof}

\section{Proof of Theorem~\ref{t.d=3}} \label{s.test}

The proof of the Theorem is by duality, namely we construct a
function $ \Psi $ of $ L ^{1}$ norm about one, which is used to
provide a lower bound on the $ L ^{\infty }$ norm of the sum of Haar
functions. The details of this argument are similar to those of
\cite{bl}.

The function $ \Psi $ will take the form of a Riesz product, but in
order to construct it, we need some definitions. Fix $
0<\varepsilon<1$ to be a small number, ultimately of order
$ 1/ d ^2 $. Define relevant parameters by
\begin{gather} \label{e.q}
q= \lfloor a n ^{\varepsilon} \rfloor \,,\qquad   b =\tfrac 14\,,
\\
\label{e.rho}
\widetilde \rho=a  q ^{b}  n^{- (d-1)/2}\,, \qquad \rho =  {\sqrt q} n ^{-(d-1)/2}.
\end{gather}
Here $ a $ is a small positive constant, we use the notation  $
b=\tfrac 14$  throughout, so as not to obscure  those aspects of the
argument that  dictate this  choice.
 $ \widetilde \rho $ is  a `false' $ L^2$
 normalization for the sums we consider, while the larger term $ \rho $ is the
 `true' $ L ^{2}$ normalization.
Our `gain over the average case estimate' in the Small Ball Conjecture is $ q ^{b}
 \simeq n ^{\varepsilon /4}$.

Divide the integers $ \{1,2,\dotsc,n\}$ into $ q$ disjoint intervals of equal length
$ I_1,\dotsc, I_q$, ordered from smallest to largest.
Let $ \mathbb A _t \coloneqq \{\vec r\in \mathbb H _n
\mid r_1\in I_t\}$.  Let
\begin{equation}
\label{e.G_t} F_t \coloneqq   \sum _{\vec r\in \mathbb A _t}  f
_{\vec r}\,, \qquad H_n \coloneqq \sum _{\abs{ R}= 2 ^{-n}} \alpha
(R) h_R.
\end{equation}
Here, the $ f _{\vec r}$ are as in \eqref{e.fr}.
 The Riesz product is a `short product':
\begin{equation*}
\Psi \coloneqq \prod _{t=1} ^{q} (1+  \widetilde  \rho F_t) \,,
\end{equation*}
One can view the $ \widetilde  \rho F_t$ as a `poor man's $
\operatorname {sgn} ( F_t)$', in that the Riesz product above tends
to weight the region where the functions $ F_t$ align. Note the
subtle way in which the false $ L^2$ normalization enters into the
product. It means that the product is, with high probability,
positive.  And of course, for a positive function $ F$, we have $
\mathbb E F=\norm F.1.$, with expectations being typically easier to
estimate. This heuristic is made precise below.

Proposition \ref{p.productsofhaars} suggests that we should
decompose the product $ \Psi $ into
\begin{equation}
\label{e.FG}  \Psi =1+\Psi ^{\textup{sd}}+\Psi ^{\neg  }\,,
\end{equation}
where the two pieces are the `strongly distinct' and `not strongly
distinct' pieces. To be specific, for integers $ 1\le u\le q$, let
\begin{equation*}
\Psi ^{\textup{sd}}_k
 \coloneqq  \widetilde \rho ^{k}
 \sum _{1\le v_1< \cdots < v_k\le q} \;
 \sideset{}{^{\textup{sd}}}\sum _{\vec r_t\in \mathbb A _{v_t}}  \prod _{t=1} ^{u} f _{\vec r
 _t}\, ,
\end{equation*}
 where $\sideset{}{^{\textup{sd}}}\sum $ is taken to
 be over all $\vec r_t\in \mathbb A_{v_t}$ $1\le{} m\le k$ such that:
 \begin{equation}  \label{e.distinct}
 \text{ the vectors  $\{\vec r _{t }\mid 1\le{} m\le k\} $ are strongly
 distinct. }
 \end{equation}
Then define
\begin{equation} \label{e.zCsd}
\Psi ^{\textup{sd}} {} \coloneqq {} \sum _{k=1} ^{q}  \Psi
^{\textup{sd}} _k.
\end{equation}

With this definition, it is clear that we have
\begin{equation} \label{e.gain>trivial}
\ip H_n, \Psi ^{\textup{sd}}, =\ip H_n, \Psi ^{\textup{sd}} _{1},
\gtrsim   q ^{b} \cdot n^{-\frac{d-1}{2}}\cdot 2 ^{-n} \sum _{\abs{
R}=2 ^{-n} } \abs{ \alpha_R }\,,
\end{equation}
so that $ q ^{b}$ is our `gain over the trivial estimate', once we
prove that  $\norm \Psi ^{\textup{sd}}.1. \lesssim 1\,$ (estimate
\eqref{e.sd1} below). Proving this inequality is the main goal of
the technical estimates of the following Lemma:

\begin{lemma}\label{l.technical} We have these estimates:
\begin{align}
\label{e.<0} \mathbb P ( \Psi <0) &\lesssim \operatorname {exp} (-A
q ^{1-2b} )\,;
\\  \label{e.expq2b}
\norm \Psi . 2. & \lesssim \operatorname {exp} (a' q ^{2b})\,;
\\
\label{e.E1}
 \mathbb E  \Psi  & = 1 \,;
\\
\label{e.CL1} \norm  \Psi .1. &\lesssim 1\,;
\\ \label{e.neq1}
\norm \Psi ^{\neg } .1. &\lesssim 1\,;
\\
\label{e.sd1} \norm \Psi ^{\textup{sd}}.1. &\lesssim 1\,.
\end{align}
Here, $ 0<a'<1$, in \eqref{e.expq2b}, is a small constant,
decreasing to zero as $ a $ in \eqref{e.q} goes to zero; and $ A>1$,
in \eqref{e.<0} is a large constant, tending to infinity as $ a$ in
\eqref{e.q} goes to zero.
\end{lemma}

\begin{proof}
We give the proof of the Lemma, assuming our main inequalities
proved in the subsequent sections.

\smallskip

\emph{Proof of \eqref{e.<0}.} Using the distributional estimate
\eqref{e.partialX2} of Theorem \ref{t.better} proved in Section 5,
and the definition of $\Psi$ we estimate

\begin{align*} \mathbb P
(\Psi <0)&\le \sum _{t=1} ^{q} \mathbb P ( \widetilde \rho \,  F_t <
-1)
\\
& = \sum _{t=1} ^{q}  \mathbb P (\rho F_t < - a ^{-1} q ^{1/2-b})
\\
& \lesssim  \operatorname {exp} (-c a ^{-2} q ^{1-2b})\,.
\end{align*}

\smallskip

\emph{Proof of \eqref{e.expq2b}.} The proof of this is detailed
enough and uses the results of subsequent sections, so we postpone
it to Section 6, Lemma \ref{l.expq2b} below.

\smallskip

\emph{Proof of \eqref{e.E1}.} Expand the product in the definition
of $ \Psi $. The leading term is one.  Every other term is a product
\begin{equation*}
\prod _{k\in V} \widetilde \rho \, F_k \,,
\end{equation*}
where $ V$ is a non-empty subset of $ \{1 ,\dotsc, q \}$.  This
product is in turn a linear combination of products of $ \mathsf r$
functions. Among each such product, the maximum in the first
coordinate is unique. This fact tells us that the expectation of
these products of $ \mathsf r$ functions is zero.  So the
expectation of the product above is zero. The proof is complete.

\smallskip

\emph{Proof of \eqref{e.CL1}.} We use the first two estimates of our
Lemma. Observe that
\begin{align*}
\norm \Psi .1. & = \mathbb E \Psi -2 \mathbb E \Psi \mathbf 1 _{\Psi
<0}
\\
& \le 1+ 2\mathbb P (\Psi <0) ^{1/2} \norm \Psi .2.
\\
& \lesssim 1+ \operatorname {exp} ( -A q ^{1-2b}/2+ a' q ^{ 2b})\,.
\end{align*}
We have taken $ b=1/4$ so that $ 1-2b=2b$.  For sufficiently small $
a$ in  \eqref{e.q}, we will have $ A \gtrsim a'$.
We see that \eqref{e.CL1} holds.

Indeed,   Lemma~\ref{l.expq2b} proves a uniform estimate, namely
\begin{equation*}
\sup _{V\subset \{1 ,\dotsc, q\}} \mathbb E \prod _{v \in V}  (1+
\widetilde \rho F_t) ^2 \lesssim \operatorname {exp} (a' q ^{2b})\,.
\end{equation*}
Hence, the argument above proves
\begin{equation}\label{e.CCLL}
\sup _{V\subset \{1 ,\dotsc, q\}} \NOrm \prod _{t \in V}  (1+
\widetilde \rho F_t) .1. \lesssim 1\,.
\end{equation}

\smallskip


\emph{Proof of \eqref{e.neq1}.}  The primary facts are
\eqref{e.CCLL} and Theorem~\ref{l.admissible<}; we use the notation
devised for that Theorem.

We use the triangle inequality,  estimate \eqref{e.expq2b} of
Lemma~\ref{l.technical}, H\"older's inequality, with indices  $ q
^{2b}$ and $\bigl( q^{2b} \bigr)'=q^{2b}/( q^{2b}-1 )$ , the
inclusion-exclusion identity \eqref{e.In-Ex} and estimate
\eqref{e.equiv} of Theorem \ref{l.admissible<} in the calculation
below. Notice that we have
\begin{align*}
\sup _{V\subset \{1 ,\dotsc, q\}} \NOrm \prod _{t \in V}  (1+
\widetilde \rho F_t) .(q ^{2b})'. &\le \sup _{V\subset \{1 ,\dotsc,
q\}} \NOrm \prod _{t \in V}  (1+ \widetilde \rho F_t) .1. ^{(q
^{2b}-1)/q ^{2b}} \times \NOrm \prod _{t \in V}  (1+ \widetilde \rho
F_t) .2. ^{2q ^{-2b}}
\\ &\lesssim 1.
\end{align*}

We now estimate
\begin{align}
\norm \Psi ^{\neg}.1. &\le \sum _{G\textup{ admissible}} \NOrm
\widetilde{\rho}\,\,^{|V(G)|} \operatorname {SumProd} (X(G)) \cdot
\prod _{t\in \{1 ,\dotsc, q\}-V(G)} (1+ \widetilde \rho F_t) .1.
\label{e.main}
\\
& \le \sum _{G\textup{ admissible}} \norm
\widetilde{\rho}\,\,^{|V(G)|} \operatorname {SumProd} (X(G)) . q
^{2b}. \cdot \NOrm  \prod _{t\in \{1 ,\dotsc, q\}-V(G)} (1+
\widetilde \rho F_t) .( q ^{2b})'. \nonumber
\\
& \lesssim \sum _{G\textup{ admissible}} \norm
\widetilde{\rho}\,\,^{|V(G)|} \operatorname {SumProd} (X(G)) . q
^{2b}. \nonumber
\\
& = \sum_{v= 2}^q \sum _{G:\, |V(G)|=v} \norm
\widetilde{\rho}\,\,^{|V(G)|} \operatorname {SumProd} (X(G)) . q
^{2b}. \nonumber
\\
& \lesssim \sum _{v=2} ^{q} \left( ^{q}_{v} \right) v^{2dv} [ q
^{C'} n ^{-\eta} ] ^{v} \nonumber
\\
& \lesssim q ^{C''} n ^{-\eta} \lesssim n ^{-\varepsilon '} \lesssim
1\,.\nonumber
\end{align}

\smallskip
\emph{Proof of \eqref{e.sd1}.}  This follows from \eqref{e.neq1} and
\eqref{e.CL1}, and the identity $ \Psi =1+\Psi ^{\textup{sd}}+\Psi
^{\neg}$ together with the triangle inequality.

\end{proof}
\section{The Analysis of the Coincidence} 

Following the language of J.~Beck \cite{MR1032337}, a
\emph{coincidence} occurs if we have two vectors $ \vec r\neq \vec
s$ with e.\thinspace g.~$ r_2=s_2$. He observed that sums over
products of $ \mathsf r$ functions in which there are coincidences
obey favorable $ L^2$ estimates.  We refer to (extensions of) this
observation as the \emph{Beck Gain.} We introduce relevant notation
for this situation. For $1\le k \le d$ and $1\le t_1,t_2 \le q$, set
\begin{equation}\label{e.Phi}
\Phi _{t_1,t_2,k} \coloneqq  \sum _{\substack{ \vec r \in \mathbb A_{t_1};\, \vec s \in \mathbb A_{t_2}\\\vec r\neq \vec s\\
r _{k}= s _{k} } } f _{\vec r} \cdot f _{\vec s}\,.
\end{equation}
Notice that due to our
 construction of the Riesz Product, there are no coincidences in the
 first coordinate in the decomposition of $\Psi$, although the case $k=1$ is important
 for the proof of the $L^2$ estimate \eqref{e.expq2b} .
In the sum above, there are $ 2d-3$ free parameters among the
vectors $ \vec r$ and
 $ \vec s$.  That is, the pair of vectors $ (\vec r,\vec s)$ are completely specified by
 their values in $ 2d-3$ coordinates. The following lemma suggests
 that these parameters behave as if they were orthogonal.

\begin{bg}\label{l.SimpleCoincie}
We have the estimates below, valid for an absolute implied constant
that is only a function of dimension $ d \ge 3$.
\begin{equation}\label{e.BG}
\sup  { \norm \Phi _{t_1, t_2,k}.p.} \lesssim   p ^{d-1/2} \cdot n
^{d-3/2}\,, \qquad  2\le p < \infty \,,
\end{equation}
where the supremum is taken over all $1\le k \le d$ and $1\le
t_1,t_2 \le q$.
\end{bg}

This estimate is smaller by $ 1/2$  power of $ n$  than what one might naively expect,
and so we say that we have an average gain of $ 1/4$ power of $ n$ in the
products above.  (Here, the average is in reference to the two functions we
form the product of.)
This Lemma, in dimension $ d=3$ appears in \cite{bl}.  We will give an inductive
proof of this estimate, that requires that we revisit the three dimensional case.
In the next  section, we also derive other estimates from the one above.

The estimate above may admit an improvement, in
that the power of $ p$ is perhaps too large by a single power, due
to our use of Proposition~\ref{p.condExpect}. (There should also
be a dependence upon $ q$, but on this point, and in many others,
the arguments of this paper are suboptimal,  and so we do not pursue this
point here.)

\begin{conjecture}\label{j.bg}
We have the estimates below, valid for an absolute implied constant
that is only a function of dimension $ d \ge 3$.
\begin{equation}\label{e.BGconj}
\sup  { \norm \Phi _{t_1, t_2,k}.p.} \lesssim   (pn) ^{d-3/2} \,,
\qquad  2\le p < \infty \,.
\end{equation}
\end{conjecture}

\subsection*{Proof of Lemma~\ref{l.SimpleCoincie}}

The proof is inductive on dimension. We shall suppress dependence on
$t_1$, $t_2$. In fact, we shall prove the
Theorem for the quantity
\begin{equation}\label{e.BGa}
\Phi _{1} \coloneqq  \sum _{\substack{ \vec r\neq \vec s \in \mathbb H_n\\
r _{1}= s _{1} } } f _{\vec r} \cdot f _{\vec s}\, ,
\end{equation}
and the claimed statement will follow with only minor adjustments.
To set up the induction, we need some definitions.

\begin{definition}\label{d.sumProd} Given a set of $ \mathsf r$ functions
$ \{f _{\vec r}\}$ and
subset $ \mathbb C \subset \mathbb H _ {n_1} \times \cdots \times \mathbb H _{n_t}$,
set
\begin{equation*}
\operatorname {SumProd} (\mathbb C )
=\sum _{(\vec r_1 ,\dotsc, \vec r_t)\in \mathbb C } \prod _{s=1} ^{t} f _{\vec r_s} \,.
\end{equation*}
Below, we will be interested in pairs and four-tuples of $ \mathsf
r$ functions. It is an important element of the argument, allowing
us to run the induction, that we consider products of $ \mathsf r$
functions where the vectors are in   hyperbolic collections $\mathbb
H _n $, for different values of $ n$.
\end{definition}

The main quantity we induct on is then
\begin{equation}\label{e.B}
\mathcal B (d,n,p) =  \sup _{\mathbb B } \norm  \operatorname {SumProd} (\mathbb B ).p.
\,,  \qquad d,n,p\ge 3\,.
\end{equation}
Here, the supremum is formed over all   $ \mathbb B \subset \mathbb
H  _{n_1} \times \mathbb H _{n_2}$ and all $ \mathsf r$ functions
subject to these conditions:

\begin{itemize}
\item  There is a coincidence in the first coordinate:
For all $ (\vec r,\vec s)\in \mathbb B $, we have $ \vec r\neq \vec
s$ and  $ r_1=s_1$.
\item $ n_1,n_2\le n$. That is the lengths of the vectors $ \vec r$ and $ \vec s$
are permitted to be different.
\item No other restriction is placed upon the pairs of vectors in $ \mathbb B $.
\end{itemize}

Our main estimate on these quantities is as follows.

\begin{lemma}\label{l.B} We have the inequality below valid for all dimensions $ d\ge3$.
\begin{equation*}
\mathcal B (d,n,p) \lesssim p ^{d-1/2} n ^{d-3/2}\,,
\qquad p,n\ge 3\,.
\end{equation*}
\end{lemma}

The inductive argument for Lemma~\ref{l.B} has the underlying strategy of
reducing dimension by
application of  the Littlewood-Paley inequalities.  But, this causes
the collections of vectors to lose some of their symmetry.  Regaining the symmetry
causes us to introduce  additional types of
collections of vectors.  Two of these collections are as follows.
\begin{equation}\label{e.C}
\mathcal C (d,n,p) =  \sup _{\mathbb C } \norm  \operatorname {SumProd} (\mathbb C ).p.
\,,  \qquad d,n,p\ge 3\,.
\end{equation}
Here, the supremum is formed over all   $ \mathbb C \subset \mathbb H  _{n_1}
\times \mathbb H _{n_2}$ and all $ \mathsf r$ functions subject to these conditions

\begin{itemize}
\item  There is a coincidence in the first coordinate:
For all $ (\vec r,\vec s)\in \mathbb C $, we have $ \vec r\neq \vec s$ and $ r_1=s_1$.
\item
For all $ (\vec r,\vec s)\in \mathbb C $, we have $ r _{2}>s_2$ and $ r_3<s_3$.
\item $ n_1,n_2\le n$.
\item There is no other restriction on the pairs of vectors in $ \mathbb C $.
\end{itemize}

The only difference between the present collections and
the collections in $ \mathcal B (d,n,p)$ is that in the present collections we assume
locations of maximums in the second and third coordinates, thereby
permitting application of the Littlewood-Paley inequalities
in those two coordinates.

The second collection is less sophisticated.  We simply assume that the maximum
always occurs in say, the first coordinate.
 Define
\begin{equation}\label{e.D}
\mathcal D (d,n,p) =  \sup _{\mathbb D } \norm  \operatorname {SumProd} (\mathbb D ).p.
\,,  \qquad d,n,p\ge 3\,.
\end{equation}
Here, the supremum is formed over all   $ \mathbb D \subset \mathbb H  _{n_1}
\times \mathbb H _{n_2}$ and all $ \mathsf r$ functions subject to these conditions

\begin{itemize}
\item  There is a coincidence in the first coordinate:
For all $ (\vec r,\vec s)\in \mathbb D $, we have $ \vec r\neq \vec s$ and $ r_1=s_1$.
\item
For all $ (\vec r,\vec s)\in \mathbb D $, and all $ 2\le j\le d$, we have $ r _{j}\ge s_j$.
\item $ n_2<n_1\le n$.
\end{itemize}
That is, we require that in each coordinate where there is a
maximum, the maximum occurs in the vector $ \vec r$.

\begin{lemma}\label{l.C+D}
 We have the inequality below valid for all dimensions $ d\ge3$.
\begin{equation*}
\mathcal C (d,n,p)\,,\, \mathcal D (d,n,p) \lesssim p ^{d-1/2} \cdot n ^{d-3/2}\,,
\qquad p,n\ge 3\,.
\end{equation*}
\end{lemma}

We turn to the proofs of the Lemma~\ref{l.B} and Lemma~\ref{l.C+D}, and begin
by explaining the logic of our induction.  Let $ \mathcal B (d)$ stand for the
inequalities in Lemma~\ref{l.B} in dimension $ d$, and likewise for $ \mathcal C (d)$
and $ \mathcal D (d)$.  We prove:

\begin{itemize}
\item The inequalities $ \mathcal D (d)$ for all dimensions $ d$.
\item The inequalities $ \mathcal B (3)$ and $ \mathcal C (3)$.  At the same time,
assuming $ \mathcal B (d-1)$, $ d\ge 4$,  we prove $ \mathcal C
(d)$.
\item
Assuming $ \mathcal C (d)$ and $ \mathcal D (d)$, we prove $
\mathcal B (d)$.
\end{itemize}

These clearly combine to prove the two Lemmas, and so complete the proof of
Lemma~\ref{l.SimpleCoincie}.

\subsection*{The Inequalities $ \mathcal D (d)$.}

The definition of $ \mathcal D (d)$ permits the possibility of equality
for a large number of coordinates of  the two vectors.
Let us exclude that case in this definition.
Define
\begin{equation}\label{e.Dneq}
\mathcal D _{\neq} (d,n,p) =  \sup _{\mathbb D }
\norm  \operatorname {SumProd} (\mathbb D ).p.
\,,  \qquad d,n,p\ge 3\,,
\end{equation}
where $ \mathbb D $ is as in \eqref{e.D}, but with the additional condition that
for $ 2\le j\le d$ we have $ r_j>s_j$.
Then, we are free to apply the Littlewood-Paley inequality in each of the
coordinates from $ 2$ to $ d$.

Fix a collection of vectors $ \mathbb D $, and a collection of
$ \mathsf r$ functions which achieves the supremum in \eqref{e.Dneq}.
For this collection, and a choice of vector $ \vec \rho \in \mathbb N ^{d-1}$, let
\begin{equation*}
\mathbb D _{\vec \rho }= \{ (\vec r, \vec s)\in \mathbb D \mid  r _{j+1}=\rho _j\,,
\ 1\le j \le d-1\}\,.
\end{equation*}
Of course there are at most $ \lesssim n ^{d-1}$ values of $ \vec\rho $ for which the
collection above is non-empty.
Then,
\begin{align*}
\mathcal D _{\neq} (d,n,p)
&\lesssim
p ^{(d-1)/2}
\NORm
\Biggl[
\sum _{\vec \rho } \operatorname {SumProd} (\mathbb D _{\vec \rho }) ^2
\Biggr] ^{1/2}
. p .
\\
& \lesssim
p ^{(d-1)/2} n ^{(d-1)/2} \sup _{\vec \rho }
\NORm
\sum _{\vec \rho } \operatorname {SumProd} (\mathbb D _{\vec \rho })  . p .  \,.
\end{align*}
But, the coordinate $  r_1$ is completely specified in $ \mathbb D
_{\vec \rho }$, and therefore does not contribute to the last norm.
And so the first coordinate of $ \vec s$ is specified.  Therefore,
there are at most $ d-2 $ free choices of parameters in the vector $
s$. By application of the Littlewood--Paley inequalities, we have
\begin{equation*}
\mathcal D _{\neq} (d,n,p)
\lesssim (pn) ^{d-3/2}\,.
\end{equation*}
This is better than the claimed inequality.

\smallskip

If there are a set $ J\subset \{2 ,\dotsc, d\}$ of coordinates for which $ r_j=s_j$
for all $ j\in J$, then after arbitrarily specifying these values, we have
will be in position to apply the inequality $ \mathcal D  _{\neq}(d- \lvert  J\rvert, n, p )$.
This will clearly give a smaller estimate.  As the number of possible choices for
$  J$ is only a function of dimension, this completes the proof.

\subsection*{The Bounds $ \mathcal B (3)$ and $ \mathcal C (3)$.
Assuming $ \mathcal B (d-1)$, $ d\ge 4$,  we prove $ \mathcal C (d)$. }

In this section, we will prove the estimates for   $ \mathcal C (3)$.
As well, we present the inductive proof of $ \mathcal C (d)$
assuming $ \mathcal B (d-1)$, for $ d\ge 4$.

\bigskip
For the proof of $ \mathcal C (3)$
there is an ancillary collection that we will have recourse to.
Let
\begin{equation}\label{e.M}
\mathcal M (n,p)=\sup _{\mathbb M} \norm \operatorname {SumProd} (\mathbb M).p.
\end{equation}
where the supremum is formed over all choices of $\mathbb M\subset \mathbb H _{n_1} \times
\mathbb H _{n_2}$ and all $ \mathsf r$ functions subject to these conditions.

\begin{itemize}
\item $ \vec r, \vec s$ are three dimensional vectors.
\item  There is a coincidence in the first coordinate:
For all $ (\vec r,\vec s)\in \mathbb C $, we have $ \vec r\neq \vec s$ and $ r_1=s_1$.
\item  The second coordinates are fixed: There are integers $ F_1, F_2$ so that
for all $ (\vec r, \vec s)\in \mathbb M $ we have $ r_2=F_1$ and $ s_2=F_2$.
\item There is no coincidence in the third coordinate:  For all $ (\vec r, \vec s) \in
\mathbb M $ we have $ r_3\neq s_3$.
\item $ n_1,n_2\le n$.
\end{itemize}

See Figure~\ref{f.M} for an illustration of this collection.We
remark that in the case $ n_1\neq n_2$, a coincidence can occur in
the third coordinate, a case that will come up below.

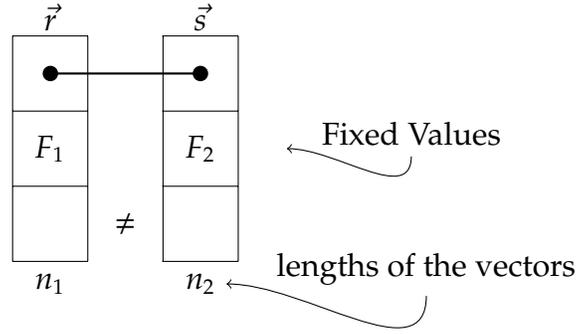
\begin{figure}
\begin{center}
 \begin{tikzpicture}
  \draw[step=1cm] (0,0) grid (1,3);
  \draw[step=1cm] (1.99,0) grid (3,3);
  \draw (.5,3.25) node {$ \vec r$};
  \draw (2.5,3.25) node {$ \vec s$};
  \draw (.5,-.3) node  (n1) {$ n_1$};
  \draw (2.5,-.3) node (n2) {$ n_2$};
  \draw (5.5,-.1) node (n5) { lengths of the vectors};
  \draw[thick,*-*] (.4,2.5) -- (2.6,2.5);
  \path[->] (n5) edge [out=-90, in=0] (n2);

  \draw (.5,1.5) node  {$ F_1$};
  \draw (2.5,1.5) node  {$ F_2$};
  \draw (3.5,1.5) node (n6) {};
    \draw (5.3,1.7) node (n7) {   Fixed Values};
   \path[->] (n7) edge [out=-90, in=0] (n6);
  \draw (1.5,.5) node {$ \neq$};
\end{tikzpicture}
\end{center}
\caption{The collections $ \mathbb M $, with a coincidence in the top row,
the second row taking fixed values, and no coincidence in the bottom row.}
\label{f.M}
\end{figure}

\begin{lemma}\label{l.M} We have the inequalities
\begin{equation}\label{e.M<}
\mathcal M (n,p) \lesssim \sqrt p \cdot \sqrt n \,.
\end{equation}
\end{lemma}

\begin{proof}
Notice that the value of the maximum in the third coordinate
completely specifies the pair of vectors $ (\vec r, \vec s)$.  Therefore, one
application of the Littlewood-Paley inequalities completes the proof.
For any collection $ \mathbb M$ as above,
let $ \mathbb M _{a}$ be the $ (\vec r, \vec s) \in \mathbb M$ where the
maximum in the third coordinate is $ a$, $ \max \{r_3,s_3\}=a$.  Note that this
can only consist, at most, of two pairs of vectors.
\begin{align*}
\norm \operatorname {SumProd} (\mathbb M).p.
\lesssim  \sqrt p
\NOrm \sum _{a} \operatorname {SumProd} (\mathbb M _a)^2 .p/2. ^{1/2}
 \lesssim \sqrt p \cdot \sqrt n \,.
\end{align*}
\end{proof}

Fix a dimension $ d\ge 3$.
Let $ \mathbb B $ be the collection which satisfies the conditions associated with
\eqref{e.B} that contains $ \mathbb C $.  We introduce a conditional expectation
into the argument, to gain some additional symmetry.
 Let
 $ \mathcal F _{a,b}$ be the dyadic sigma field in the second and third coordinates
 generated by dyadic rectangles of side lengths $ 2 ^{-a-1}$ and $ 2 ^{-b-1}$ respectively.

We have this equality.
\begin{equation}\label{e.==}
 \sum _{\substack{ (\vec r,\vec s)\in \mathbb C \\ r_2=a\,,\, s_3=b  }}
f _{\vec r} \cdot f _{\vec s}
=
\mathbb E \Bigl(
 \sum _{\substack{ (\vec r,\vec s)\in \mathbb B \\ r_2=a\,,\, s_3=b  }}
f _{\vec r} \cdot f _{\vec s} \,\vert\, \mathcal F _{a,b}\Bigr) -
\operatorname {SumProd} ( \mathbb D _{a,b}),
\end{equation}
where $ \mathbb D _{a,b}$ consists of pairs of vectors $ (\vec r, \vec s)\in
\mathbb B $ such that $r_1=s_1 $, $ a=r_2=s_2$ and $ b=r_3=s_3$.
In three dimensions, the set $ \mathbb D _{a,b}$ is empty, since the requirements
for a pair of vectors being in the set $ \mathcal D _{a,b}$
forces $ \vec r=\vec s$, a contradiction.

Assuming that $ d>3$, using the assumption of $ \mathcal B (d-2)$ (
in the case of $d=4$ we just apply the Littlewood-Paley inequality
in the last coordinate), we see that
\begin{equation}\label{e.CDab}
\norm \operatorname {SumProd} ( \mathbb D _{a,b}) .p/2. \lesssim p
^{d-5/2} \cdot  n ^{d-7/2}\,.
\end{equation}
Here, we have `lost two dimensions' due to the roles of $ a, b$.
Therefore, using a trivial estimate in the parameters $ a,b$,
\begin{equation*}
p\NORm \Biggl[ \sum _{a,b} \operatorname {SumProd} ( \mathbb D
_{a,b}) ^2    \Biggr] ^{1/2} .p. \lesssim p ^{ d -3/2} n ^{d-5/2}\,.
\end{equation*}
This estimate is smaller than what the other terms  will give us.

\smallskip

Therefore, using \eqref{e.FS} we can estimate
\begin{equation}\label{e.3c}
\norm \operatorname {SumProd} (\mathbb C ).p. \lesssim p ^{ d -3/2}
n ^{d-5/2} + p ^2  \NORm \sum _{a,b} \ABs{ \sum _{\substack{ (\vec
r,\vec s)\in \mathbb B \\ r_2=a\,,\, s_3=b  }} f _{\vec r} \cdot f
_{\vec s} } ^2 . p/2. ^{1/2}.
\end{equation}
We concentrate on the latter term, and in particular expand the
square.
\begin{align}\label{e.CC1}
\sum _{a,b} \ABs{ \sum _{\substack{ (\vec r,\vec s)\in \mathbb B \\
r_2=a\,,\, s_3=b  }} f _{\vec r} \cdot f _{\vec s} } ^2 & \lesssim n
^{2d-3}
\\& \label{e.CC2}\quad +
 \operatorname {SumProd} (\mathbb B'_1)+\operatorname {SumProd} (\mathbb B'_2)
\\& \label{e.CC3}\quad +
\operatorname {SumProd} (\mathbb B'')
\end{align}
where these terms arise as follows.  In forming the square on the
left in (\ref{e.CC1}), we have two pairs $ (\vec r,\vec s), (\vec
{\underline r},\vec {\underline s}) \in \mathbb C $ with $ r_2=
r'_2$ and $ s_3=s'_3$. We form the product
\begin{equation} \label{e.cF}
f _{\vec r} \cdot f _{\vec s} \cdot f _{\vec {\underline r}} \cdot f
_{\vec {\underline s}}
\end{equation}
\begin{itemize}
\item If the two pairs are equal, the product  in (\ref{e.cF}) is one.  There are $ \lesssim  n ^{2d-3}$
ways to select such pairs.  This is the right hand side of
(\ref{e.CC1}).
\item The collection $\mathbb B'_1$ consists of vectors such that $ \vec r=\vec {\underline r}$ but
$ \vec s\neq \vec {\underline s}$, the product in (\ref{e.cF}) is
equal to $ f _{\vec s} \cdot f _{\vec {\underline s}}$ ($\mathbb
B'_2$ is defined symmetrically). Notice that necessarily we have $
s_1=s'_1$, which is equal to $r_1$, and $ s_3=s'_3$. Let us set
\begin{equation*}
\mathbb B'_c = \{ (\vec s,\vec {\underline s})\mid s_1=\underline
s_1 = c;\, s_3=\underline s_3\}\,.
\end{equation*}
 We have `lost' one parameter in $\mathbb B'_c$ and have one more coincidence, therefore, we can apply the induction hypothesis
  $ \mathcal B (d-1)$ to see that
\begin{equation*}
\norm \operatorname {SumProd} (\mathbb B'_c ) .p. \lesssim p
^{d-3/2} n ^{d-5/2}\,.
\end{equation*}
It is easy to see that
\begin{equation*}
\operatorname{SumProd}(\mathbb B'_1 ) = \sum_{\vec r \in \mathbb
H_{n_1}} \operatorname{SumProd}(\mathbb B'_{r_1} ).
\end{equation*}
Thus we have
\begin{equation*}
\Norm \operatorname{SumProd}(\mathbb B'_1 ).p.  \le \sum_{\vec r \in
\mathbb H_{n_1}} \Norm \operatorname{SumProd}(\mathbb B'_{r_1} ).p.
\le n^{d-1} \cdot p ^{d-3/2} n ^{d-5/2} =p ^{d-3/2} n ^{d-7/2}\,.
\end{equation*}

 This controls the  term  in (\ref{e.CC2}).

\item The last term arises from two pairs of vectors
$ (\vec r,\vec s), (\vec {\underline r},\vec s) \in \mathbb C$ that consist of four distinct vectors.
Let us set
\begin{equation*}
\mathbb B''
=
\{ (\vec r,\vec s,\vec {\underline s},\vec {\underline r})\mid
(\vec r,\vec s), (\vec {\underline r},\vec {\underline s}) \in \mathbb C\,,\ \vec r\neq \vec {\underline r}\,,\
\vec s\neq \vec {\underline s} \}
\end{equation*}
Here, for the sake of cleaner graphics, we have deliberately written $ \vec s, \vec {\underline s}$
as the middle two vectors in the four-tuples in $ \mathbb B''$.
\end{itemize}

 \begin{figure}
 \begin{center}
  \begin{tikzpicture}
  \draw[step=1cm] (0,0) grid (1,4);
  \draw[step=1cm] (1.99,0) grid (4,4);
  \draw[step=1cm] (4.99,0) grid (6,4);
  \draw (.5,4.25) node {$ \vec r$};
  \draw (2.5,4.25) node {$ \vec s$};
  \draw (3.5,4.25) node {$ \vec {\underline s}$};
  \draw (5.5,4.25) node {$ \vec {\underline r}$};
  \draw (.5,-.3) node  (n1) {$ n_1$};
  \draw (2.5,-.3) node (n2) {$ n_2$};
  \draw (3.5,-.3) node  (n3){$ n_2$};
  \draw (5.5,-.3) node (n4) {$ n_1$};

  \draw (9,-.3) node (n5) {lengths of the vectors};
 \path[->] (n5) edge [out=-90, in=0] (n4);
  \draw[thick,*-*] (.4,3.5) -- (2.6,3.5) node [midway, below] {$ {}_{F_1}$};
  \draw[thick,*-*] (3.4,3.5) -- (5.6,3.5) node [midway, below] {$ {}_{F_2}$};
  \draw[thick,*-*] (2.4,1.5) -- (3.6,1.5);
 \path[thick,*-*] (.4,2.5) edge [out=-15,in=195] (5.6,2.5);
  \draw[thick] (2.5,-.6) .. controls (2.8, -.8) and (3.2, -.8) ..  (3.5, -.6)
  node [midway,below] {inside};
  \draw[thick] (.5,-.7) .. controls (2.5, -1.8) and (3.5, -1.8) ..  (5.5, -.7)
  node [midway,below] {outside};
   \end{tikzpicture}

  \end{center}
 \caption{The Decomposition of $ \mathbb B'' _{F_1,F_2}$, in the four dimensional case.
 Note that the coincidences are indicated by the connected black circles.}
  \label{f.FF}
 \end{figure}
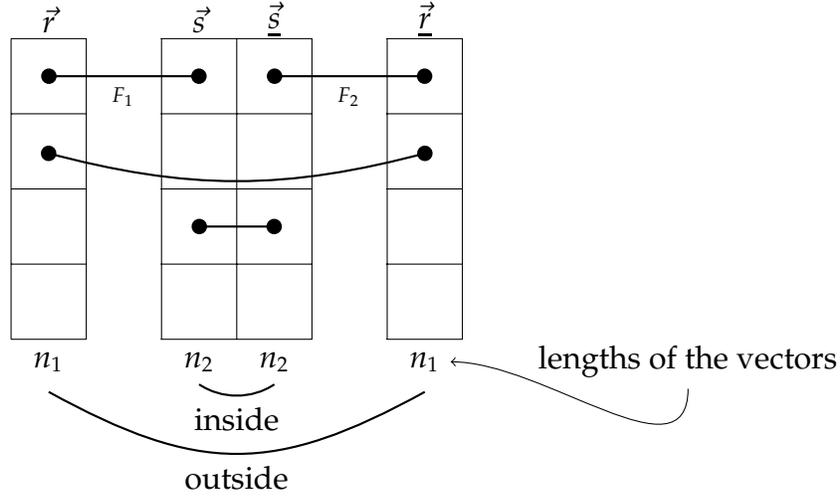

It remains to bound the term in \eqref{e.CC3}.  We reduce this
four-fold product back to a product of two-fold products. For
integers $ F_1, F_2$, let $ \mathbb B _{F_1,F_2} '' $ be those $
(\vec r,\vec s,\vec {\underline s},\vec {\underline r})\in \mathbb
B''$ with $ r_1=s_1=F_1$ and $ {\underline r}_1={\underline
s}_1=F_2$.   Let
$
\mathbb B'' _{\textup{outside}, F_1, F_2}
$
be the projection of four-tuples in  $ \mathbb B _{F_1,F_2} ''$ onto the
first and fourth coordinates, and $\mathbb B'' _{\textup{inside}, F_1, F_2}  $
the projection onto the second and third coordinates.
See Figure~\ref{f.FF}.


For any pair $ (\vec r,\vec {\underline r})\in \mathbb B'' _{\textup{outside}, F_1, F_2} $,
and any two pairs
\begin{equation*}
(\vec s,\vec {\underline s})\,,\; (\vec \sigma ,\vec {\underline \sigma})
\in
\mathbb B'' _{\textup{inside}, F_1, F_2}\,,
\end{equation*}
we have
\begin{equation*}
(\vec r,\vec s, \vec {\underline s},\vec {\underline r})\,,\;
(\vec r,\vec \sigma , \vec {\underline \sigma},\vec {\underline r})
\in \mathbb B'' _{F_1,F_2}\,.
\end{equation*}
Therefore, we have the product formula
\begin{equation*}
\operatorname {SumProd} (\mathbb B _{F_1,F_2} '')
=
\operatorname {SumProd} (\mathbb B'' _{\textup{outside}, F_1, F_2} )
\times
\operatorname {SumProd} (\mathbb B'' _{\textup{inside}, F_1, F_2} )\,.
\end{equation*}

Notice that the pairs of vectors in $ \mathbb B''
_{\textup{outside}, F_1, F_2} $ have their first coordinates fixed,
and have a coincidence in the second coordinate. The fixed first
coordinates need not be the same, so that the lengths of the
remaining coordinates are, in general, distinct.  Still, we  may
conclude that
\begin{equation*}
\norm
\operatorname {SumProd} (\mathbb B'' _{\textup{outside}, F_1, F_2} ) .p.
\lesssim
p ^{d-3/2} n ^{d-5/2}\,.
\end{equation*}
This estimate is uniform in $ F_1,F_2$.  In the case of dimension $ d=3$,
this follows from Lemma~\ref{l.M}, while for $ d>3$ it follows from the
induction hypothesis.
A similar inequality holds for
$ \mathbb B'' _{\textup{inside}, F_1, F_2} $.

Therefore,
we can estimate the term in \eqref{e.CC3} as follows:
\begin{align*}
\NOrm  \operatorname {SumProd} (\mathbb B'') . p/2. ^{1/2} &
\lesssim p n \sup _{F_1,F_2} \norm \operatorname {SumProd} (\mathbb
B'' _{\textup{outside}, F_1, F_2} ) \times \operatorname {SumProd}
(\mathbb B'' _{\textup{inside}, F_1, F_2} ) .p/2. ^{1/2}
\\
& \lesssim p n \sup _{F_1,F_2} \norm \operatorname {SumProd}
(\mathbb B'' _{\textup{outside}, F_1, F_2} ) .p. ^{1/2}\times \norm
\operatorname {SumProd} (\mathbb B'' _{\textup{inside}, F_1, F_2} )
.p.^{1/2}
\\ &
\lesssim (pn) ^{d-3/2}\,.
\end{align*}
Our proof is complete.  Assuming $ \mathcal B (d-1)$, $ d\ge 4$,  we
have proved $ \mathcal C (d)$.  We have also proved $ \mathcal C
(3)$. The fact that $ \mathcal B (3)$ holds follows from the
argument below.

\subsection*{Assuming $ \mathcal C (d)$ and $ \mathcal D (d)$, we prove $ \mathcal B (d)$. }

Fix $ p,n\ge 3$, a collection of  vectors $ \mathbb B $ and
$ \mathsf r$ functions which achieve the supremum in \eqref{e.B}.  Write this collection as
\begin{equation*}
\mathbb B =\mathbb D \cup \bigcup _{2\le i\neq j\le d} \mathbb C _{i,j}
\end{equation*}
where $ \mathbb C _{i,j}$ consists of those pairs $ (\vec r,\vec s)\in \mathbb B $
such that $ i$ is the first coordinate for which $ r_i>s_i$ and $ j$ is the first
coordinate for which $ r_j<s_j$.  Then, the collections $ \mathbb C _{i,j}$
are pairwise disjoint, and the collection $ \mathbb D $ consists of all pairs
not in some $ \mathbb C _{i,j}$.
Thus,
\begin{equation*}
\operatorname {SumProd} (\mathbb B )
=
\operatorname {SumProd} (\mathbb D )
+
\sum _{2\le i\neq j\le d} \operatorname {SumProd} (\mathbb C _{i,j}) \,.
\end{equation*}

After a harmless permutation of indices, the inequalities $ \mathcal C (d)$
apply to the collections $ \mathbb C _{i,j}$.  The (unconditional) inequalities
$ \mathcal D$ apply to the collection $ \mathbb D $.  The proof is complete.

\section{Corollaries of the Beck Gain}

 Theorem~\ref{t.LP} implies an exponential estimate of order
 $ \operatorname {exp} (L ^{2/ (d-1)})$ for sums of $ \vec r$ functions.
In fact, we can derive a \emph{subgaussian} estimate for such sums, for moderate deviations,
and moreover, in order to have a gain of order $ n ^{c/d ^2 }$ in our Main Theorem, we need
to use this estimate.

\begin{theorem}\label{t.better} Using the notation of \eqref{e.rho} and \eqref{e.G_t},
we have this estimate, valid for all $ 1\le t\le q$.
\begin{equation}\label{e.better}
\norm  \rho F_t  .p.  \lesssim  \sqrt p \,  \,, \qquad 1\le p \le c
n ^{\frac {1 - 2 \varepsilon} {2d-1}} \,.
\end{equation}
As a consequence, we have the distributional estimate
\begin{equation}\label{e.partialX2}
\mathbb P ( \lvert  \rho F _t\rvert  >x ) \lesssim \operatorname
{exp} (- c x ^2 )\,, \qquad x< c  n ^{  \frac {1- 2\varepsilon }
{4d-2}}\,.
\end{equation}
Here $ 0<c<1$ is an absolute constant.
\end{theorem}

To use \eqref{e.partialX2},  we need $ q ^{b}= a ^{b} n ^{\epsilon
\cdot b} < c n ^{  \frac {1 } {4d-6}}$, and so $ \epsilon \simeq 1/d$
is the optimal value for $ \epsilon $ that this proof will give.

\begin{proof}

Recall that
\begin{equation*}
F_t =  \sum _{\vec r\in \mathbb A _t}  f _{\vec r}\,.
\end{equation*}
where  $ \mathbb A _t \coloneqq \{\vec r\in \mathbb H _n \mid r_1\in
I_t\}$, and $ I_t$  in an interval of integers of length $ n/q$, so
that $ \sharp \mathbb A _t \simeq   n ^{d-1} /q \simeq \rho ^{-2}$.

Apply the Littlewood-Paley inequality  in the first coordinate.
This results in the estimate
\begin{align*}
\norm \rho F_t .p.
& \lesssim \sqrt p
\NOrm \Bigl[\sum _{s\in I_j} \Abs{ \rho \sum _{\vec r\,:\, r_1=s} f _{\vec r}} ^2
\Bigr] ^{1/2} .p.
\\
& \lesssim  \sqrt p \norm   1+   \rho ^{2}\Phi _{t,t,1}   .p/2.
^{1/2}
\\
& \lesssim  \sqrt p \Bigl\{ 1 + \norm \rho ^{2}\Phi _{t,t,1}   .p/2.
^{1/2}\Bigr\} \,,
\end{align*}
where $\Phi _{t,t,1}$ is defined in \eqref{e.Phi}. Here it is
important to use the constants in the Littlewood-Paley inequalities
that give the correct order of growth of $ \sqrt p$. Of course the
terms $ \Phi _{t,t,1}$ are controlled by the estimate in
\eqref{e.BG}. In particular, we have
\begin{equation}\label{e.FFFjjj}
\norm \rho ^{2}\Phi _{t,t,1}.p. \lesssim  \frac{ q } {n ^{d-1}}  p
^{d-1/2}  n ^{d-3/2}   \lesssim q\, p ^{d-1/2} n ^{-1/2} \lesssim
1\,.
 \end{equation}
Hence \eqref{e.better} follows.

\smallskip

The second distributional inequality is a well known consequence of the norm
inequality.  Namely, one has the inequality below, valid for all $ x$:
\begin{equation*}
\mathbb P (\rho F_t >x ) \le C ^{p} p ^{p/2} x ^{-p}\,, \qquad 1\le
p \le c  n ^{\frac {1 -2\varepsilon} {2d-1}} \,.
\end{equation*}
If $ x$ is as in \eqref{e.partialX2}, we can take $ p \simeq x ^{2}$ to prove the
claimed exponential squared bound.
 \end{proof}


We shall now use the Beck Gain to prove the crucial $L^2$ estimate
\eqref{e.expq2b} of Lemma \ref{l.technical}. We actually need a
slightly more general inequality:
\begin{lemma}\label{l.expq2b}
We have the following estimate:
\begin{equation}
\sup _{V\subset \{1 ,\dotsc, q\}} \mathbb E \prod _{v \in V}  (1+
\widetilde \rho F_t) ^2 \lesssim \operatorname {exp} (a' q ^{2b})\,.
\end{equation}
\end{lemma}

The supremum over $ V$ will be an immediate consequence of the proof
below, and so we don't address it specifically.

\begin{proof}[Proof of \eqref{e.expq2b}.]

Let us give the  essential initial observation. We expand
\begin{equation*}
\mathbb E \prod _{j=1} ^{q} (1+ \widetilde \rho F_j) ^2
=
\mathbb E \prod _{j=1} ^{q} (1+ 2\widetilde \rho F_j+ (\widetilde \rho F_j) ^2  )\,.
\end{equation*}
Hold the last $d-1$ coordinates, $ x_2,\dots, x_d$, fixed and let $
\mathcal F$ be the sigma field generated by $ F_1 ,\dotsc, F_{q-1}$.
We have
\begin{equation} \label{e.;p}
\begin{split}
\mathbb E \bigl(1+ 2\widetilde \rho F_q+ (\widetilde \rho F_q) ^2
\,\big|\, \mathcal F\bigr) &=1+\mathbb E \bigl((\widetilde \rho F_q) ^2
\,\big|\, \mathcal F\bigr)
\\
&=1+   a ^2  q ^{2b-1}+ \widetilde \rho ^2  \mathbb E (\Phi _{q,q,1} \,\big|\, \mathcal F)\,,
 \end{split}
\end{equation}
where $ \Phi _{q,q,1}$ is defined in \eqref{e.Phi}.
 Then, we see that
\begin{align} \nonumber
\mathbb E \ \prod _{v=1} ^{q} (1+ 2\widetilde \rho F_t+ (\widetilde
\rho F_t) ^2  ) &= \mathbb E \Bigl\{  \prod _{v=1} ^{q-1} (1+
2\widetilde \rho F_t+ (\widetilde \rho F_t) ^2  )\, \times \mathbb E
\bigl(1+ 2\widetilde \rho F_t+ (\widetilde \rho F_t) ^2 \,\big|\,
\mathcal F \bigr)\Bigr\}
\\
&\le  \label{e.;;}
(1+a ^2 q ^{2b-1})
\mathbb E \prod _{v=1} ^{q-1} (1+ 2\widetilde \rho F_t+ (\widetilde \rho F_t) ^2  )
\\  \label{e.;;;;}
& \qquad + \mathbb E \abs{ \widetilde\rho ^{2} \Phi _{q,q,1} } \cdot
\prod _{v=1} ^{q-1} (1+ 2\widetilde \rho F_t+ (\widetilde \rho F_t)
^2  )
\end{align}
This is the main observation: one should induct on \eqref{e.;;},
while treating the term in \eqref{e.;;;;} as an error, as the Beck
Gain estimate \eqref{e.BG} applies to it.

Let us set up notation to implement this line of approach.  Set
\begin{equation*}
N (V;r) \coloneqq \NOrm \prod _{t=1} ^{V} (1+ \widetilde \rho F_t)
.r. \,, \qquad   V=1 ,\dotsc, q\,.
\end{equation*}
We will obtain a very crude estimate for these numbers for $ r=4$.
Fortunately, this is relatively easy for us to obtain. Namely,
$ q$ is small enough that we can use the inequalities \eqref{e.better} to see that
\begin{align*}
N (V; 4)&\le
\prod _{v=1} ^{V} \norm 1+ \widetilde \rho F_t .4V.
\\
& \le ( 1 +  C q ^{1/2+b}  ) ^{V}
\\
& \le (Cq) ^{   q}\,.
\end{align*}
We have the estimate below from H\"older's inequality
\begin{equation}\label{e.killq^q}
N (V;2(1-1/ q) ^{-1} )\le N (V;2) ^{1-1/ q} \cdot N (V; 4) ^{1/
q}\,.
\end{equation}

We see that \eqref{e.;;}, \eqref{e.;;;;} and \eqref{e.killq^q}
give us the inequality
\begin{equation}\label{e.=.=}
\begin{split}
N (V+1; 2) ^2  & \le (1+a ^2 q ^{2b-1})    N (V; 2) ^2 + C  \cdot  N
(V; 2 (1-1/q) ^{-1} ) ^2   \cdot \norm  \widetilde \rho ^{2} \Phi
_{V+1,V+1,1} .  q.
\\
& \le (1+a ^2 q ^{2b-1})    N (V; 2)^2 + C  N (V;2) ^{2-2/q} \cdot N
(V; 4) ^{2/ q} \norm  \widetilde \rho ^{2} \Phi _{V+1,V+1,1} .  q.
\\
& \le (1+a ^2 q ^{2b-1})   N (V; 2) ^2 + C   q ^{d+2 } n ^{-1/2} N
(V;2) ^{2-2/ q} \,.
\end{split}
\end{equation}
In the last line we have used the inequality \eqref{e.BG}. Of
course we only apply this as long as $ N (V; 2)\ge 1$.  Assuming
this is true for all $ V\ge 1$, we see that
\begin{equation*}
N (V+1; 2) ^2 \le (1+a ^2 q ^{2b-1} + C   q ^{d+2 } n ^{-1/2} ) N
(V; 2) ^2 \,.
\end{equation*}
And so, by induction,
\begin{align*}
N (q;2)&\lesssim  (1+a ^2  q ^{2b-1}+  C   q ^{d+2 } n ^{-1/2} )
^{q/2} \lesssim \operatorname e ^{ 2a\, q ^{2b}}\,.
\end{align*}
Here, the last inequality will be true for large $ n$, provided that
$\varepsilon$ in the definition of $q$ \eqref{e.q} is small. Indeed,
we need
\begin{align*}
a ^2  q ^{2b-1}&\ge   C   q ^{d+2} n ^{-1/2}
\end{align*}
Or equivalently,
\begin{equation*}
a ^2 n ^{1/2}\gtrsim q ^{d+5/2 }\,.
\end{equation*}
Comparing to the definition of $ q$ in \eqref{e.q}, we see that the
proof is finished.
\end{proof}

One should notice that the results of this section suggest that our
methods give a gain of the order $\frac1d$.

\section{The Beck gain with fixed parameters.}

We will need to analyze longer products of $ \mathsf r$ functions.
These longer products   will be reduced to the case of a  a slightly
more general version of the Beck Gain Lemma \ref{l.SimpleCoincie}.
Namely, we will consider sums of products of two $ \mathsf r$
fucntions, but impose the additional restriction for some
coordinates in the pair of vectors to have  fixed values. Let $\vec
a \in \mathbb N^{F_1}$ and  $\vec b \in \mathbb N^{F_2}$ be integer
vectors with lengths $|\vec a|, |\vec b| <n$. We will be estimating
the quantity:

\begin{equation}\label{e.Bf}
\mathcal B (F_1, F_2) =  \sup_{\vec a, \vec b, j_1<j_2} \sup
_{\mathbb B } \norm \operatorname {SumProd} (\mathbb B ).p. \,,
\qquad d,n,p\ge 3\,.
\end{equation}
The inner supremum is formed over all   $ \mathbb B \subset \mathbb
H _{n} \times \mathbb H _{n}$ and all $ \mathsf r$ functions subject
to these conditions:

\begin{itemize}
\item $\vec r \in \mathbb A_{j_1}$, $\vec s \in \mathbb A_{j_2}$,
where $j_1<j_2$ (i.e. $s_1$ is the maximum in the first coordinate.)
\item  There is a coincidence in the second coordinate:
For all $ (\vec r,\vec s)\in \mathbb B $, we have $ \vec r\neq \vec
s$ and  $ r_2=s_2$.
\item For $k=1,\dots,F_1$, we have $r_{k+2} = a_k$. ($F_1$
coordinates of $\vec r$ are fixed.)
\item For $k=1,\dots,F_2$, we have $s_{F_1+k+2} = b_k$. ($F_2$
coordinates of $\vec s$ are fixed, and these coordinates
are distinct from the other vector.)
\end{itemize}

We have the following estimate, which gives an average  Beck Gain
of $ n ^{1/8}$ for each of the two functions in the product.

\begin{lemma}\label{l.Bf} We have the inequality below valid for all dimensions $ d\ge3$.
\begin{equation*}
\mathcal B (F_1, F_2) \lesssim p ^{d- 1 -
\frac{F_1+F_2}{2}-\frac{1}{4}} n ^{d- 1 -
\frac{F_1+F_2}{2}-\frac{1}{4}}\,, \qquad p,n\ge 3\,.
\end{equation*}
\end{lemma}

\begin{proof} We will reduce this situation to the Beck Gain proven
before. Let $\mathbb B$ be as above. First of all, we shall apply
the Littlewood-Paley inequality in the first coordinate. Notice that
the maximum in this coordinate is automatically $s_1$.

\begin{equation}\label{e.3cf}
\norm \operatorname {SumProd} (\mathbb B ).p. \lesssim
 \sqrt{p}   \NORm \sum _{c\in I_{j_2}} \ABs{ \sum
_{\substack{ (\vec r,\vec s)\in \mathbb B \\  s_1=c  }} f _{\vec r}
\cdot f _{\vec s} } ^2 . p/2. ^{1/2}
\end{equation}
We concentrate on the latter term, and in particular expand the
square.
\begin{align}\label{e.3cf1}
 \sqrt{p}   \NORm \sum _{c\in \mathbb I_{j_2}} \ABs{ \sum
_{\substack{ (\vec r,\vec s)\in \mathbb B \\  s_1=c  }} f _{\vec r}
\cdot f _{\vec s} } ^2 . p/2. ^{1/2} & =  \sqrt{p}   \NORm  \sum
_{\substack{ (\vec r,\vec s, \vec {\underline r}, \vec {\underline s})\in \mathbb B\times \mathbb B \\
s_1=\underline s_1 }} f _{\vec r} \cdot f _{\vec s}\cdot
f_{\vec{\underline r}} \cdot f_{\vec {\underline s}}  . p/2.
^{1/2}\\
\label{e.1} & \le \sqrt{p} n \max_{c\neq \underline c} \NORm  \sum
_{\substack{ (\vec r,\vec s, \vec {\underline r}, \vec {\underline s})\in \mathbb B\times \mathbb B \\
s_1=\underline s_1;\, r_2=s_2=c;\, \underline r_2= \underline
s_2=\underline c }} f _{\vec r} \cdot f _{\vec s}\cdot
f_{\vec{\underline r}} \cdot f_{\vec {\underline s}}  . p/2.
^{1/2}
\\\label{e.2} & + \sqrt{p} \sqrt{n} \max_{c} \NORm  \sum
_{\substack{ (\vec r,\vec s, \vec {\underline r}, \vec {\underline s})\in \mathbb B\times \mathbb B \\
s_1=\underline s_1;\, r_2=s_2= \underline r_2= \underline s_2= c }}
f _{\vec r} \cdot f _{\vec s}\cdot f_{\vec{\underline r}} \cdot
f_{\vec {\underline s}}  . p/2. ^{1/2}
\end{align}

We start with the estimates for the first term above \eqref{e.1}:

\begin{align*}
& \sqrt{p} n \max_{c\neq \underline c} \NORm  \sum
_{\substack{ (\vec r,\vec s, \vec {\underline r}, \vec {\underline s})\in \mathbb B\times \mathbb B \\
s_1=\underline s_1;\, r_2=s_2=c;\, \underline r_2= \underline
s_2=\underline c }} f _{\vec r} \cdot f _{\vec s}\cdot
f_{\vec{\underline r}} \cdot f_{\vec {\underline s}}  . p/2.
^{1/2}\\
&= \sqrt{p} n \max_{c\neq \underline c} \NORM  \Biggl(\sum _{ (\vec
r, \vec {\underline r})\in \mathbb B_1} f _{\vec r} \cdot
f_{\vec{\underline r}} \Biggr) \times  \Biggl(\sum _{ (\vec s, \vec
{\underline s})\in \mathbb B_2} f _{\vec s} \cdot f_{\vec{\underline
s}} \Biggr)  . p/2. ^{1/2}\\
&\le \sqrt{p} n \max_{c\neq \underline c} \NORm  \sum _{ (\vec r,
\vec {\underline r})\in \mathbb B_1} f _{\vec r} \cdot
f_{\vec{\underline r}} .p.^{1/2} \NOrm \sum _{ (\vec s, \vec
{\underline s})\in \mathbb B_2} f _{\vec s} \cdot f_{\vec{\underline
s}} . p. ^{1/2}
\end{align*}
Here $\mathbb B_1$ is defined to consist of pairs $(\vec r, \vec
{\underline r}) \in \mathbb A _{j_1}^2$ which satisfy the following:
\begin{itemize}
\item For $k=1,\dots,F_1$, we have $r_{k+2}= \underline r_{k+2} = a_k$.
\item $r_2=c$, $\underline r_2 =\underline c$.
\end{itemize}
And similarly $\mathbb B_2$ consists of pairs $(\vec s, \vec
{\underline s}) \in \mathbb A _{j_2}^2$ with the properties:
\begin{itemize}
\item For $k=1,\dots,F_2$, we have $s_{k+F_1+2}= \underline s_{k+F_1+2} = b_k$.
\item $s_2=c$, $\underline s_2 =\underline c$.
\item Moreover, we have $s_1=\underline s_1$.
\end{itemize}

Notice that because of the last condition and the fact that $c\neq
\underline c$ (i.e., $\vec s \neq \vec{\underline s}$), the Beck
Gain (Lemma \ref{l.SimpleCoincie}) applies to this family of pairs,
giving a gain of $ n ^{1/2}$, while $\mathbb B_1$ will be estimated
by simple parameter counting, supplying no gain. We have
\begin{align*}
\Norm \operatorname{SumProd}(\mathbb B_1).p. &\lesssim
(pn)^{d-2-F_1} ,
\\
 \Norm
\operatorname{SumProd}(\mathbb B_2).p. &\lesssim p^{d-3/2-F_2}
n^{d-2-F_2-1/2} \,.
\end{align*}
 And thus we can estimate the term \eqref{e.1} by
\begin{align*}
\sqrt{p} n  \NORm  \sum
_{\substack{ (\vec r,\vec s, \vec {\underline r}, \vec {\underline s})\in \mathbb B\times \mathbb B \\
s_1=\underline s_1;\, r_2=s_2=c;\, \underline r_2= \underline
s_2=\underline c }} f _{\vec r} \cdot f _{\vec s}\cdot
f_{\vec{\underline r}} \cdot f_{\vec {\underline s}}  . p/2. ^{1/2}
& \lesssim \sqrt{p} n \left( (pn)^{d-2-F_1}\right)^{1/2} \left(
p^{d-3/2-F_2} n^{d-2-F_2-1/2}\right)^{1/2}\\
& = (pn)^{d-1-\frac{F_1+F_2}{2}-\frac14}.
\end{align*}

The second term \eqref{e.2} satisfies the same bound in $n$. This
can be shown by simple parameter counting, the gain comes from the
loss of one parameter since $c=\underline c$.

We remark that in this version of the Beck gain `error terms' do not
arise, since we apply Littlewood-Paley inequality only in the first
coordinate, where we already have a natural order. Thus we do not
need to use the conditional expectation argument as in the proof of
Lemma \ref{l.SimpleCoincie}.

\end{proof}

\section{The Beck Gain for Longer Coincidences} \label{s.nsd}



 In the present section we treat longer coincidences. This requires
  a careful analysis of the variety
of ways that a product can fail to be strongly distinct. That is, we
need to understand the variety of ways that coincidences can arise,
and how coincidences can contribute to a smaller  norm.
Following Beck, we will use the language of Graph Theory to describe
these general patterns of coincidences.

\subsection*{Graph Theory Nomenclature}

We adopt familiar nomenclature from Graph Theory, although there is
no graph theoretical fact that we need, rather the use of this
language is just a convenient way to do some bookkeeping. The class
of graphs that we are interested in satisfies particular properties.
A \emph{$d-1$ colored graph} $ G$ is the tuple $ (V (G), E_2,
E_3,\dots,E_{d})$, of the \emph{vertex set} $ V (G)\subset \{1
,\dotsc, q\}$, and \emph{ edge sets $ E_2,\,  E_3,\,\dots\, E_{d}$,
of colors $ 2, 3,\dots, d$ respectively}. Edge sets are are subsets
of
\begin{equation*}
E _{j}\subset V (G) \times V (G) - \{ (k,k)\;|\;  k\in V (G)\}\,.
\end{equation*}
Edges are symmetric, thus if $ (v,v')\in E_j$ then necessarily $
(v',v)\in E_j$.

A \emph{clique of color $ j$} is a maximal subset $ Q\subset V (G)$
such that for all $ v\neq v'\in Q$ we have $ (v,v')\in E_j$.  By
\emph{maximality}, we mean that no strictly larger set of vertices
$ Q'\supset Q$ satisfies this condition.

Call a graph $ G$ \emph{admissible} iff
\begin{itemize}
\item  The edges sets, in all $d-1$ colors, decompose into a union of cliques.
\item  If  $ Q_k$'s are  cliques of color $ k$ ($k=2, \dots, d$),
then $\bigcap_{k=2}^d Q_k $ contains at most one vertex.
\item Every vertex is in at least one clique.
\end{itemize}

 A graph $ G$ is \emph{connected } iff for any two vertices
in the graph, there is a path that connects them. A \emph{path} in
the graph $ G$ is a sequence of vertices  $ v_1 ,\dotsc, v_k$ with
an edge of \emph{any color,} spanning adjacent vertices , that is $
(v _{j}, v _{j+1}) \in  \cup_{k=2}^d E_k $.

\subsection*{Reduction to Admissible Graphs}

It is clear that admissible graphs as defined above are naturally
associated to sums of products of $ \mathsf r$ functions. Given
admissible graph $ G$ on vertices $ V$, we set $ X (G)$ to be those
tuples of $ \mathsf r$ vectors
\begin{equation*}
\vec r _v \in \prod _{v\in V} \mathbb A _v
\end{equation*}
so that if $ (v, v')$ is an edge of color $ j$ in $ G$, then $ r
_{v,j}= r _{v',j}$.

We shall introduce the following counting parameter: for an
admissible graph $G$, its index, $ind(G)$, is defined as
\begin{equation}\label{e.ind}
ind(G)= \sum_{Q\textup{ is a clique}} \left( \sharp Q - 1\right).
\end{equation}
Effectively, the index of $G$ is the least number of equalities,
needed to define $X(G)$, in other words, the number of coincidences.
In particular, for the graphs, corresponding to the simplest case of
the Beck Gain, the index is one.

With these definitions at hand, it is not hard to obtain the
Inclusion-Exclusion formula, relating admissible graphs and the `not
strongly distinct' part of the Riesz product:

\begin{equation}\label{e.In-Ex}
\Psi^{\neg} = \sum_{G\textup{ admissible}} (-1)^{ind(G)+1}\,
\widetilde{\rho}\,\,^{|V(G)|} \operatorname{SumProd}(X(G))\cdot
\prod_{t\notin V(G)} (1+\widetilde{\rho} F_t).
\end{equation}

 We will prove the following Theorem:
\begin{theorem}\label{l.admissible<}{\bf Beck Gain for Graphs} For an admissible graph $ G$ on vertices $ V$ we have
the estimate below for positive, finite constants $ C_0, C_1, C_2,
C_3$:
\begin{equation}\label{e.equiv}
\rho ^{\abs{ V}} \norm \operatorname {SumProd} (X (G)).p.
 \le  [C_0\abs{ V} ^{C_1} p ^{C_2} q ^{C_3}   n ^{-\eta }] ^{ \abs{ V}}\,,
\qquad 2<p< \infty  \,.
\end{equation}
\end{theorem}

The most significant term on the right is $ n ^{-\eta }$.  It shows
that as the number of coincidences goes up, the corresponding `Beck
Gain' improves. Notice that for the other terms on the right, $ C_0$
is a constant; $ \lvert  V\rvert\le q \le n ^{\epsilon } $, where we
can choose $ 0<\epsilon $ as a function of $ \eta $; and while the
inequality above holds for all $ 2\le p<\infty $, we will only need
to apply it for $ p \lesssim  q ^{2b}\le n ^{\epsilon /2}$. That is,
the $ n ^{-\eta }$ is the dominant term on the right. This Theorem,
together with the fact that there are at most $|V|^{2d|V|}$
admissible graphs on the vertex set $V$,  yields the boundedness of
the sum in \eqref{e.main}.

\subsection*{Norm Estimates for Admissible Graphs}

We begin the proof of Theorem \ref{l.admissible<} with a further
reduction to connected admissible graphs. Let us write $ G\in
\operatorname {BG} (C_0, C_1,C_2, C_3 ,\eta )$ if the estimates
\eqref{e.equiv} holds. (`$ \operatorname {BG}$' for `Beck Gain.') We
need to see that all admissible graphs are in $ \operatorname {BG}
(C_0, C_1,C_2, C_3 ,\eta)$ for non-negative, finite choices of the
relevant constants.

\begin{lemma}\label{l.holder} Let $ C_0, C_1,C_2, C_3, \eta$
be non-negative constants. Suppose that $ G$ is an admissible graph,
and that it can be written as a union  of subgraphs $ G_1 ,\dotsc,
G_k$ on disjoint vertex sets, where all $ G_j \in \operatorname {BG}
(C_0, C_1,C_2, C_3, \eta)$. Then,
\begin{equation*}
G \in \operatorname {BG} (C_0, C_1, C_2, C_2+C_3, \eta )\,.
\end{equation*}
\end{lemma}

With this Lemma, we will identify a small class of graphs for which
we can verify the property \eqref{e.equiv} directly, and then appeal
to this Lemma to deduce Lemma~\ref{l.admissible<}.   Accordingly, we
modify our notation. If $ \mathcal G$ is a class of graphs, we write
$ \mathcal G\subset \operatorname {BG} (\eta )$ if there are
constants $ C_0, C_1, C_2, C_3 $ such that $ \mathcal G\subset
\operatorname {BG} (C_0, C_1, C_2, C_3, \eta )$.

\begin{proof}
We then have  by Proposition~\ref{p.products}
\begin{equation*}
\operatorname {SumProd} (X (G)) = \prod _{j=1} ^{k} \operatorname
{SumProd} (X (G_j)) \,.
\end{equation*}
Using H\"older's inequality, we can estimate
\begin{align*}
\rho^{|V|}\norm \operatorname {SumProd} (X (G)) .p. &\le \prod
_{j=1} ^{k} \rho^{|V_j|}\norm \operatorname {SumProd} (X (G_j)). k
p.
\\
& \le \prod _{j=1} ^{k}  [C_0 (kp) ^{C_1} q ^{C_2} n ^{-\eta } ]
^{\abs{ V_j}}
\\
& \le [C_0 p ^{C_1} q ^{C_2+C_1} n ^{-\eta } ] ^{\abs{ V}}\,.
\end{align*}
Here, we use the fact that since the graphs are non-empty, we
necessarily have $ k\le q$.

\end{proof}

\begin{proposition}\label{p.products}
Let $ G_1 ,\dotsc, G_p$ be admissible graphs on pairwise disjoint
vertex sets $ V_1 ,\dotsc, V_p $.  Extend these graphs in the
natural way to a graph $ G$ on the vertex set $ V=\bigcup V_t$.
Then, we have
\begin{equation*}
\operatorname {SumProd} (X(G)) = \prod _{t=1} ^{p} \operatorname
{SumProd} (X (G_t))\,.
\end{equation*}
\end{proposition}

\subsection*{Connected Graphs Have the Beck Gain.}

We single out for special consideration the connected  admissible
graphs $ G$ .   Let $ \mathcal G _{\textup{connected}} $ be the
collection of of all admissible  connected
 graphs on $ V\subset \{1,\dotsc, q\}$.

\begin{lemma}\label{l.twoCliques}
We have $ \mathcal G _{\textup{connected}}\subset \operatorname {BG}
( {\eta} )$ for some $\eta>0$.
\end{lemma}

The point of this proof is that we will reduce this question to
a much simpler key fact, namely  Lemma~\ref{l.Bf}, which we restate here in our current
notation.\footnote{The only points that recommend the proof we describe here
is that it is easy to state and delivers a gain.
Clearly, a more sustained analysis, yielding a larger gain
would result in an improved result on the Small Ball Conjecture.}

Let $ \mathcal G _{\textup{fixed}} (2)$ be the set of graphs---and sets of
$ \mathsf r$ functions associated with the graphs---with these properties:
\begin{itemize}
\item $ G$ is a connected graph on two vertices $ \{v,v'\}$. That is, there is
at least one edge that connects these to vertices. Denote by
$C\subset \{2,\dots,d\}$ the set of coordinates corresponding to the
edges.
\item There are a set of coordinates $ F_v, F_{v'}\subset \{2 ,\dotsc, d\}$
that are disjoint from the set of edges, and two vectors $ \vec a\in \mathbb N ^{F_v}$
and $ \vec a ' \in \mathbb N ^{F_{v'}}$, so that we define
\begin{equation*}
Y (G):= \{ (\vec r_v,\vec r _{v'})\in \mathbb H _n\mid
r _{v,j}=r _{v',j}\ \forall j\in C\,; \
r _{v,k}=a _{k}\ \forall k\in F_{v}\,; \
r _{v',k}=a _{k}\ \forall k\in F_{v'}
\}
\end{equation*}
\end{itemize}
These are in essence the assumptions of Lemma~\ref{l.Bf}. This Lemma
proves that
\begin{equation*}
\norm \operatorname {SumProd} (Y (G)). p.
\lesssim p ^{d}
n ^{\sigma }\,, \qquad \sigma =d- 1 - \frac{F_v+F_{v'}}{2}-\frac{1}{4}\,.
\end{equation*}

By abuse of notation, let us summarize this inequality by the inclusion
$ \mathcal G _{\textup{fixed}} (2)\subset \operatorname {BG} (C_0,C_1, d/2,0, 1/8 )$.
Or, even more briefly, as
$ \mathcal G _{\textup{fixed}} (2)\subset \operatorname {BG} (1/8 )$.
That is, there is a gain of $ \tfrac18$ for each vertex.
It follows from the proof of Lemma~\ref{l.holder}, that if $ G$ is any graph
whose connected components are each elements of $ \mathcal G _{\textup{fixed}} (2)$,
then $ G\in \operatorname {BG} (1/8)$.

Our line of attack on this Lemma is to take a general connected
graph $ G$, use the triangle inequality to assign fixed values to a
number of edges, making the connected components of the new graph to
be elements of $ \mathcal G _{\textup{fixed}} (2) $. The proportion
of vertices that will be in one of these graphs will be at least $
1/2d$ of all vertices. And therefore connected graphs will be in $
\operatorname {BG} (1/16d)$.

\begin{remark}\label{r.}  A heuristic guides this argument.  The normalization
$ \rho ^{\lvert  V\rvert }$ in \eqref{e.equiv} assigns a weight $ n
^{-1/2}$ to each free parameter of $ X (G)$, ignoring losses of
parameters from the edges of $ G$.   If $ (v,v')$ is an edge in the
graph, and we assign the edge one of $ n$ possible values, the full
power of $ n$ is exactly compensated by the collective weight of the
two parameters in the edge. Therefore, we are free to fix a fixed
proportion of edges in the graph, obtaining a Beck Gain on the
remaining proportion.   In this argument, if the edge is in a clique
of size at least $ k\ge 3$, specifying a single value on this clique
actually leads to a positive gain of $ n ^{-k/2+1}$. In other words,
graphs, all of whose cliques are of size two, are extremal with
respect to this analysis (see Lemma \ref{l.g2}). This heuristic is
made precise in the proof below.

\end{remark}

By 'deleting a clique' we shall mean fixing a value of the
coincidence which corresponds to that clique. Let $G\in \mathcal G
_{\textup{connected}}$. Following the heuristic above, in the first
step of the algorithm we delete all cliques of size at least 3 in
$G$.

After this step $G$ breaks down into connected components, which are
admissible graphs with cliques only of size 2 (and, possibly, some
singletons). Next, we want to obtain an estimate for such graphs.

\begin{lemma}\label{l.g2}
Suppose $\widetilde{G}\in \mathcal G _{\textup{connected}}$ has
cliques of size at most 2. Then $\widetilde{G}\in
BG(\frac{1}{16d})$.
\end{lemma}

To prove this statement we shall use the following property of
$\widetilde{G}$:
\begin{itemize}
\item The degree of each vertex in $\widetilde{G}$ is at most $d-1$
(since the degree in each color is at most one).
\end{itemize}

Let $\widetilde V$ be the set of vertices of $\widetilde G$, and
$\widetilde E$ be the set of all its edges. The point is to select a
maximal  subset $ \widetilde E _{\textup{indpndt}}$ of
\emph{independent} edges. That is, no two edges in $ \widetilde E
_{\textup{indpndt}}$, regardless of color, have a common vertex. It
is an elementary fact that we can take
\begin{equation} \label{e.2d-4}
\lvert  \widetilde E _{\textup{indpndt}}\rvert \ge \tfrac 1 {2d-3}
\lvert  \widetilde E\rvert\,.
\end{equation}
Indeed, each edge in $ \widetilde G$ shares a vertex with at most $
2d-4$ distinct edges, which observation directly implies the
inequality above.

We delete all other edges of $\widetilde G$ (i.e. we fix some choice
of parameters for the corresponding coincidences) and thus
$\widetilde G$ breaks down into a number of components each of which
is either a singleton or a graph with two vertices and one edge. The
latter components correspond exactly to the situation in which the
Beck gain of the previous section is applicable. Let us denote these
pairs by $G_k \in \mathcal G _{\textup{fixed}} (2) $,
$k=1,\dots,N=\lvert  \widetilde E _{\textup{indpndt}}\rvert$; the
singletons -- by $v_j$, $j=1,\dots, |\widetilde V | - 2N$. Let also
$E'=\widetilde E - \widetilde E_{\textup{indpndnt}}$ denote the set
of all deleted edges in $\widetilde G$. Denote also by $F_k$ the
number of fixed parameters in $X(G_k)$ and $F'_j$ will be the number
of fixed parameters in $\vec r_{v_j}$. We have the following
relations:
\begin{equation}\label{e.e}
2|E'| = 2|\widetilde E - \widetilde
E_{\textup{indpndnt}}|=\sum_{k=1}^N F_k +\sum_{j=1}^{|\widetilde V |
- 2N} F'_j,
\end{equation}

and, since $ \widetilde G$ is connected, it has at least $ \lvert  V
(G)\rvert -1$ edges, thus

\begin{equation}\label{e.n}
N \ge \frac{|\widetilde E|}{2d-3} \ge \frac{|\widetilde V|-1}{2d-3}
\ge \frac{|\widetilde V|}{2(2d-3)} \ge \frac{|\widetilde V|}{4d}.
\end{equation}
Besides, by Proposition \ref{p.products}, we obtain the following
equality (the sum below is taken over all choices of parameters on
the `deleted' edges):
\begin{equation}
\operatorname {SumProd} (X (\widetilde G)) = \sum \prod _{k=1} ^{N}
\operatorname {SumProd} (X ( G_k)) \cdot \prod _{j=1} ^{|\widetilde
V|-2N} \operatorname {SumProd} (X ( v_j)) \,.
\end{equation}

Now we apply the triangle inequality, H\"{o}lder's inequality, the
relations \eqref{e.e} and \eqref{e.n}, and the Beck gain in the form
of Lemma \ref{l.Bf} to estimate ($ \kappa = |\widetilde V|-N < q$):

\begin{align*}
\rho^{|\widetilde V|}\norm \operatorname {SumProd} (X (\widetilde
G)) .p. &\le n^{|E'|}\cdot \prod _{k=1} ^{N} \rho^{2}\norm
\operatorname {SumProd} (X (G_k)). \kappa p. \cdot \prod _{j=1}
^{|\widetilde V| -2N} \rho \norm  f_{\vec r_{v_j}}. \kappa p.
\\
& \lesssim n^{|E'|}\cdot \prod _{k=1} ^{N} \left[ \rho^2 (\kappa p\,
n) ^{d-1 -\frac{F_k}{2}-\frac14} \right] \cdot \prod _{j=1}
^{|\widetilde V|-2N} \left[ \rho (\kappa p \, n) ^{\frac{d-1}{2}
-\frac{F'_j}{2}}   \right]
\\
& \lesssim \left[C p ^{\frac{d-1}{2}} q ^{\frac{d}{2}}\right]
^{\abs{ \widetilde V}}\cdot n^{-\frac{N}{4}} \lesssim \left[ p
^{\frac{d-1}{2}} q ^{\frac{d}{2}}\, n^{-\frac{1}{16d}} \right]
^{\abs{ \widetilde V}}.
\end{align*}

This proves Lemma \ref{l.g2}. The point of passing to the collection
of independent edges is that $\operatorname {SumProd} (X (\widetilde
G))$ splits into a product of terms associated with graphs in $
\mathcal G _{\textup{fixed}} (2)$. Each of these graphs leads to a
gain of at least $ \tfrac 18$ for each vertex.  But by \eqref{e.e},
there are at least $ \frac1{2d} \lvert  V (G)\rvert $ vertices  for
which we will get this gain. This shows that $G\in BG(1/16d)$.

 We can now proceed to prove Lemma
\ref{l.twoCliques} -- the proof will be in the same spirit. After we
delete "large" (of size at least 3) cliques of $G$, this graph
decomposed into some singletons and components as in Lemma
\ref{l.g2} (but with some parameters fixed). Denote these components
by $\widetilde G_k$, $k=1,\dots,n_1$ and the singletons by $u_j$,
$j=1,\dots,n_2$. Let $f_k$ be the number of fixed parameters in
$X(\widetilde G_k)$ and and $f'_j$ -- the number of fixed parameters
in $\vec r_{u_j}$. Notice that the proof of Lemma \ref{l.g2} can be
trivially adapted to the case when some parameters are fixed to
obtain the estimate:
\begin{equation}\label{e.g2f}
\rho ^{\abs{\widetilde  V_k}} \norm \operatorname {SumProd} (X
(\widetilde G_k)).p.
 \le  \left[ C p ^{\frac{d-1}{2}} q ^{\frac{d}{2}}   n ^{- \frac{1}{16d}
 }\right]  ^{ \abs{\widetilde V_k}}\, n^{-\frac{f_k}{2}}.
\end{equation}
Also, if we denote by $K$ the total number of fixed cliques, one can
see that, since all the cliques had size at least 3, we have the
inequality:
\begin{equation}\label{e.k}
3K \le \sum_{k=1}^{n_1} f_k + \sum_{j=1}^{n_2} f'_j.
\end{equation}
Let us write the set of vertices of $G$ as $V=V_1 \cup V_2$, where
$V_1$ are the vertices involved in at least one of the deleted
cliques and $V_2$ are all the other vertices. It is easy to see that
$V_2 \subset \cup_{k=1}^{n_1} V(\widetilde G_k)$. Indeed, all the
vertices that became singletons had to be a part of one of the
deleted cliques. Thus,
\begin{equation}
|V_2| \le \sum_{k=1}^{n_1} |V(\widetilde G_k)|.
\end{equation}
Besides, it is easy to see that
\begin{equation}
|V_1| \le \sum_{k=1}^{n_1} f_k + \sum_{j=1}^{n_2} f'_j,
\end{equation}
because at least one parameter is fixed in each vertex from a
deleted clique.  Using these relations, similarly to the proof of
Lemma \ref{l.g2}, taking $\kappa = n_1 +n_2<q$, we can write:
\begin{align*}
\rho^{| V|}\norm \operatorname {Prod} (X ( G)) .p. &\le n^{K}\cdot
\prod _{k=1} ^{n_1} \rho^{|V(\widetilde G_k)|}\norm \operatorname
{Prod} (X (\widetilde G_k)). \kappa p. \cdot \prod _{j=1} ^{n_2}
\rho \norm f_{\vec r_{u_j}}. \kappa p.
\\
& \lesssim n^{K}\cdot \prod _{k=1} ^{n_1} \left[C p ^{\frac{d-1}{2}}
q ^{{d}}   n ^{- \frac{1}{16d}
 }\right]  ^{ \abs{V(\widetilde G_k)}}\, n^{-\frac{f_k}{2}} \cdot \prod _{j=1}
^{n_2} \left[  p^{\frac{d-1}{2}} q^d n^{-\frac{f'_j}{2}} \right]
\\
& \lesssim \left[ C p ^{\frac{d-1}{2}} q ^{{d}}\right] ^{\abs{V}}
\cdot n^{K-\frac12 \left(\sum f_k + \sum f'_j\right) -
\frac{1}{16d}\sum \abs{V(\widetilde G_k)}}
\\
& \lesssim \left[ C p ^{\frac{d-1}{2}} q ^d\right] ^{\abs{  V}}\,
n^{-\frac16 \abs{V_1}-\frac{1}{16d}\abs{V_2}} \lesssim \left[ C p
^{\frac{d-1}{2}} q ^d n^{-\frac{1}{16d}}\right] ^{\abs{ V}}.
\end{align*}

\section{The Lower Bound on the Discrepancy Function} 

We give the proof of Theorem~\ref{t.discrep}, which is essentially a corollary
to the proof of our Main Theorem, Theorem~\ref{t.d=3}. As such, we will give
a somewhat abbreviated proof. Indeed, the analogy between the lower bound on
Discrepancy Functions and the Small Ball Inequality is well known to experts.

The proof is by duality.
Fix $ N$, and take $ 2N\le 2 ^{n}<4N$.  It is a familiar fact \cite {MR903025}
that for each $ \lvert  \vec r\rvert =n$ we can construct a $ \mathsf r$ function
$ f _{\vec r}$ such that
\begin{equation}\label{e.IP}
\ip D_N, f _{\vec r}, >c>0\,,
\end{equation}
where $ c$ depends only on dimension.  We use these functions in the construction of
the test function, following \S~\ref{s.test}, with this one change.  Before,
see \eqref{e.G_t}, we took $ I_1,\dotsc, I_q$ to be a partition of
 $ \{1,2,\dotsc,n\}$ into $ q$ disjoint intervals of equal length.
Instead, we take
\begin{equation}\label{e.Iq}
I _{t} := \{ j\in \mathbb N \mid  \lvert  j- tn/q\rvert < q/4  \}\,.
\end{equation}
This is the only change we make in the construction of $ \Psi ^{\textup{sd}}$.
It follows that $ \norm \Psi ^{\textup{sd}}.1. \lesssim 1$.

Recall that $ \Psi ^{\textup{sd}}=\sum _{k=1} ^{q} \Psi ^{\textup{sd}}_k$,
see \eqref{e.distinct}.   By construction, we have
\begin{align*}
\ip D_N, \Psi ^{\textup{sd}}_1, & = \sum _{t=1} ^{q}
\widetilde \rho \sum _{\vec r\in \mathbb A _t} \ip D_N, f _{\vec r},
\\& \gtrsim q ^{b} n ^{ (d-1)/2} \simeq n ^{\epsilon /4+ (d-1)/2}\,.
\end{align*}
This is a `gain over the average case estimate' as one can see by comparison
to Theorem~\ref{t.roth}.  It remains to see that the higher order terms
$ \Psi ^{\textup{sd}}_k$ contribute smaller terms than the one above.

By construction, $\Psi ^{\textup{sd}}_k$ is itself a sum of $
\mathsf r$ functions $ f _{\vec s}$ with $ \lvert  \vec s\rvert >n$.
Indeed, it follows from the separation in \eqref{e.Iq} that we
necessarily have
\begin{equation}\label{e.Ds}
n+k\tfrac n {2q}\le \lvert  \vec s\rvert \le nd\,.
\end{equation}
Second, it is a well known fact that $ \lvert  \ip D_N, f _{\vec
s},\rvert < N 2 ^{- \lvert  \vec s\rvert }$.  Third, we fix $ \vec
s$ as above, and set $ \operatorname {Count} (\vec s)$ to be the
number of distinct ways can we select $ \vec r_1 ,\dotsc, \vec r_k$,
all of length $ n$, so that the product $ f _{\vec r_1} \cdots f
_{\vec r_k}$ is an $ \mathsf  r$ function of parameter $ \vec s$.  A
very crude bound here is sufficient,
\begin{equation*}
\operatorname {Count} (\vec s)\le \lvert \vec s \rvert ^{(d-1)k}\,.
\end{equation*}
Thus, we can estimate
\begin{align*}
\ip D_N, \Psi ^{\textup{sd}}_k, & \le \sum _{ j \ge n+k\tfrac n
{2q}}\left( \sum_{\vec s : \,  \lvert \vec s\rvert=j }\operatorname
{Count} (\vec s) \lvert \ip D_N, f _{\vec s} ,\rvert \right)
\\& \lesssim n ^{d(k+3)} 2 ^{-kn/2q}\,.
\end{align*}
As $ q= n ^{\epsilon }$, this is clearly summable in $ k\ge 1$ to at most a constant.
This completes the proof.

\section{The Proof of the Smooth Small Ball Inequality} 

We prove Theorem~\ref{t.smooth}. There is no loss of generality in
assuming that $|\alpha(R)|\le 1$ for all $R$ of volume at least
$2^{-n}$, since both sides of \eqref{e.smoothEtaTalagrand} are
homogeneous and sums have finitely many terms. With $ \varphi $ as
in the theorem, set
\begin{equation*}
\varphi _{\vec r}= \sum _{R\mid \lvert  R_j\rvert=2 ^{-r_j} } \alpha (R) \varphi _{R}\,.
\end{equation*}
And let $ \Phi =\sum _{\lvert  \vec r\rvert=n } \varphi _{\vec r}$.
Define the $ \mathsf r$ functions as in \eqref{e.fr}.
It is the assumption that $ c _{\varphi }=\ip \varphi , h _{[-1/2,1/2]},\neq 0$, and in fact we
will assume that this inner product is positive.
Thus,
\begin{equation}\label{e.zIP}
\ip \varphi _{\vec r}, f _{\vec r},=c _{\varphi } 2 ^{-n}\sum _{R\mid \lvert  R_j\rvert= 2 ^{r_j} }
\lvert  \alpha (R)\rvert \,.
\end{equation}

 As $ \varphi \in C [-1/2,1/2]$, we have
\begin{equation} \label{e.IPphi}
\lvert  \ip \varphi , h_J, \rvert \le C _{\varphi } \lvert  J\rvert
\end{equation}
for all dyadic intervals $ J$.

It is important to note that
\begin{equation}\label{e.varphiIP}
\lvert  \ip \varphi _{\vec r}, f _{\vec s},\rvert
\lesssim
\begin{cases}
0 &  \exists j \mid s_j<r_j
\\
C _{\varphi }2 ^{- \lvert  \vec r-\vec s\rvert } & \textup{otherwise}
\end{cases}
\end{equation}
The first line follows from the fact that $ \varphi $ is supported on $ [-1/2,1/2]$,
so that if e.\thinspace g.\thinspace $s_1<r_1 $, the fact that $ \varphi $ has mean
zero proves this estimate.  The second estimate follows from \eqref{e.IPphi} and
the assumption that the coefficients $ \alpha (R)$ are at most one in absolute value.

Let us take the intervals $ I_t$ in \eqref{e.Iq}, and let us assume that
\begin{equation} \label{e.4}
\sum _{\lvert  R\rvert=2 ^{n} } \lvert  \alpha (R)\rvert
\le 4 \sum _{t=1} ^{q} \sum _{\vec r\in \mathbb A _t}
\sum _{R\mid \lvert  R_j\rvert=2 ^{-r_j} }\lvert  \alpha (R)\rvert \,.
\end{equation}
If this inequality fails, it is an easy matter to redefine the $ I_t$ so that
the inequality above is true, and adjacent intervals $ I_t,I _{t+1}$ are
seperated by $ n/q$.

We then follow \S~\ref{s.test} as before to define our test function $ \Psi ^{\textup{sd}}$.
It follows that $ \norm \Psi ^{\textup{sd}}.1. \lesssim 1$.
Using \eqref{e.4}, \eqref{e.zIP} and \eqref{e.varphiIP}, we have
\begin{align*}
\ip  \Phi , \Psi ^{\textup{test}}_1,
&\ge c 2 ^{-n} \widetilde \rho
\sum _{t=1} ^{q} \sum _{\vec r\in \mathbb A _t}
\sum _{R\mid \lvert  R_j\rvert=2 ^{r_j} }\lvert  \alpha (R)\rvert
\\
& \gtrsim 2 ^{-n } n ^{ -(d-1)/2+\epsilon /4}
\sum _{\lvert  R\rvert=2 ^{-n} } \lvert  \alpha (R)\rvert \,.
\end{align*}
This is the main term.

It remains to see that the inner products $ \lvert  \ip \Phi , \Psi ^{\textup{sd}}_k, \rvert $
are small $ k\ge 1$.  The details of this calculation are very similar to the corresponding
calculuations in the previous section, hence they are omitted.


\begin{bibsection}
 \begin{biblist}

\bib{MR1032337}{article}{
    author={Beck, J{\'o}zsef},
     title={A two-dimensional van Aardenne-Ehrenfest theorem in
            irregularities of distribution},
   journal={Compositio Math.},
    volume={72},
      date={1989},
    number={3},
     pages={269\ndash 339},
      issn={0010-437X},
    review={MR1032337 (91f:11054)},
}



\bib{MR903025}{book}{
    author={Beck, J{\'o}zsef},
    author={Chen, William W. L.},
     title={Irregularities of distribution},
    series={Cambridge Tracts in Mathematics},
    volume={89},
 publisher={Cambridge University Press},
     place={Cambridge},
      date={1987},
     pages={xiv+294},
      isbn={0-521-30792-9},
    review={MR903025 (88m:11061)},
}


\bib{bl}{article}{
 author={Bilyk, Dmitriy},
 author={Lacey, Michael T.},
 title={On the Small  Ball  Inequality in Three Dimensions},
  eprint={arXiv:math.CA/0609815},
  journal={Duke Math J., to appear},
  date={2006},
 }


 \bib{MR1777539}{article}{
   author={Dunker, Thomas},
   title={Estimates for the small ball probabilities of the fractional
   Brownian sheet},
   journal={J. Theoret. Probab.},
   volume={13},
   date={2000},
   number={2},
   pages={357--382},
   issn={0894-9840},
   review={\MR{1777539 (2001g:60085)}},
}


 \bib{2000b:60195}{article}{
    author={Dunker, Thomas},
    author={K{\"u}hn, Thomas},
    author={Lifshits, Mikhail},
    author={Linde, Werner},
     title={Metric entropy of the integration operator and small ball
            probabilities for the Brownian sheet},
  language={English, with English and French summaries},
   journal={C. R. Acad. Sci. Paris S\'er. I Math.},
    volume={326},
      date={1998},
    number={3},
     pages={347\ndash 352},
      issn={0764-4442},
    review={MR2000b:60195},
}

\bib{MR1439553}{article}{
    author={Fefferman, R.},
    author={Pipher, J.},
     title={Multiparameter operators and sharp weighted inequalities},
   journal={Amer. J. Math.},
    volume={119},
      date={1997},
    number={2},
     pages={337\ndash 369},
      issn={0002-9327},
    review={MR1439553 (98b:42027)},
}

\bib{MR637361}{article}{
   author={Hal{\'a}sz, G.},
   title={On Roth's method in the theory of irregularities of point
   distributions},
   conference={
      title={Recent progress in analytic number theory, Vol. 2},
      address={Durham},
      date={1979},
   },
   book={
      publisher={Academic Press},
      place={London},
   },
   date={1981},
   pages={79--94},
   review={\MR{637361 (83e:10072)}},
}


\bib{MR94j:60078}{article}{
    author={Kuelbs, James},
    author={Li, Wenbo V.},
     title={Metric entropy and the small ball problem for Gaussian measures},
   journal={J. Funct. Anal.},
    volume={116},
      date={1993},
    number={1},
     pages={133\ndash 157},
      issn={0022-1236},
    review={MR 94j:60078},
}


\bib{MR2003m:60131}{article}{
    author={K{\"u}hn, Thomas},
    author={Linde, Werner},
     title={Optimal series representation of fractional Brownian sheets},
   journal={Bernoulli},
    volume={8},
      date={2002},
    number={5},
     pages={669\ndash 696},
      issn={1350-7265},
    review={MR 2003m:60131},
}


\bib{MR1861734}{article}{
   author={Li, W. V.},
   author={Shao, Q.-M.},
   title={Gaussian processes: inequalities, small ball probabilities and
   applications},
   conference={
      title={Stochastic processes: theory and methods},
   },
   book={
      series={Handbook of Statist.},
      volume={19},
      publisher={North-Holland},
      place={Amsterdam},
   },
   date={2001},
   pages={533--597},
   review={\MR{1861734}},
}



\bib{MR0066435}{article}{
   author={Roth, K. F.},
   title={On irregularities of distribution},
   journal={Mathematika},
   volume={1},
   date={1954},
   pages={73--79},
   issn={0025-5793},
   review={\MR{0066435 (,575c)}},
}

\bib{MR0319933}{article}{
   author={Schmidt, Wolfgang M.},
   title={Irregularities of distribution. VII},
   journal={Acta Arith.},
   volume={21},
   date={1972},
   pages={45--50},
   issn={0065-1036},
   review={\MR{0319933 (47 \#8474)}},
}

\bib{MR0491574}{article}{
   author={Schmidt, Wolfgang M.},
   title={Irregularities of distribution. X},
   conference={
      title={Number theory and algebra},
   },
   book={
      publisher={Academic Press},
      place={New York},
   },
   date={1977},
   pages={311--329},
   review={\MR{0491574 (58 \#10803)}},
}



\bib{MR95k:60049}{article}{
    author={Talagrand, Michel},
     title={The small ball problem for the Brownian sheet},
   journal={Ann. Probab.},
    volume={22},
      date={1994},
    number={3},
     pages={1331\ndash 1354},
      issn={0091-1798},
    review={MR 95k:60049},
}

\bib{MR96c:41052}{article}{
    author={Temlyakov, V. N.},
     title={An inequality for trigonometric polynomials and its application
            for estimating the entropy numbers},
   journal={J. Complexity},
    volume={11},
      date={1995},
    number={2},
     pages={293\ndash 307},
      issn={0885-064X},
    review={MR 96c:41052},
}



\bib{MR1005898}{article}{
   author={Temlyakov, V. N.},
   title={Approximation of functions with bounded mixed derivative},
   journal={Proc. Steklov Inst. Math.},
   date={1989},
   number={1(178)},
   pages={vi+121},
   issn={0081-5438},
   review={\MR{1005898 (90e:00007)}},
}

\bib{MR1018577}{article}{
   author={Wang, Gang},
   title={Sharp square-function inequalities for conditionally symmetric
   martingales},
   journal={Trans. Amer. Math. Soc.},
   volume={328},
   date={1991},
   number={1},
   pages={393--419},
   issn={0002-9947},
   review={\MR{1018577 (92c:60067)}},
}




  \end{biblist}
 \end{bibsection}
 \end{document}